\documentclass[12pt,reqno]{amsart}
\usepackage{amsfonts}
\usepackage{amssymb,mathrsfs,color}
 \usepackage{color}
\definecolor{deepgreen}{cmyk}{1,0,1,0.5}

\def\norm#1{\|#1\|}

\newcommand{\les}{{\lesssim}}

\newcommand{\F}{\mathcal{F}}

\newcommand{\cS}{\mathcal{S}}

\newcommand{\cP}{\mathcal{P}}

\newcommand{\K}{\mathcal{K}}

\newcommand{\C}{\mathbb{C}}
\newcommand{\N}{\mathbb{N}}
\newcommand{\R}{\mathbb{R}}
\newcommand{\Z}{\mathbb{Z}}

\newcommand{\al}{\alpha}
\newcommand{\be}{\beta}

\newcommand{\de}{\delta}
\newcommand{\e}{\varepsilon}
\newcommand{\fy}{\varphi}
\newcommand{\om}{\omega}
\newcommand{\la}{\lambda}
\newcommand{\te}{\theta}

\newcommand{\ka}{\kappa}
\newcommand{\x}{\xi}
\newcommand{\y}{\eta}

\newcommand{\ro}{\rho}

\newcommand{\De}{\Delta}
\newcommand{\Om}{\Omega}

\newcommand{\p}{\partial}
\newcommand{\na}{\nabla}

\newcommand{\supp}{\operatorname{supp}}

\newcommand{\lec}{\lesssim}
\newcommand{\gec}{\gtrsim}

\newcommand{\etc}{,\ldots,}
\newcommand{\I}{\infty}

\newcommand{\ti}{\widetilde}
\newcommand{\ba}{\overline}

\newcommand{\LR}[1]{{\langle #1 \rangle}}

\newcommand{\EQ}[1]{\begin{equation}\begin{split} #1 \end{split}\end{equation}}
\newcommand{\Del}[1]{}
\newcommand{\CAS}[1]{\begin{cases} #1 \end{cases}}
\newcommand{\mat}[1]{\begin{pmatrix} #1 \end{pmatrix}}
\newcommand{\pt}{&}
\newcommand{\pr}{\\ &}
\newcommand{\pq}{\quad}
\newcommand{\pn}{}

\numberwithin{equation}{section}

\newtheorem{thm}{Theorem}[section]
\newtheorem{cor}[thm]{Corollary}
\newtheorem{lem}[thm]{Lemma}

\theoremstyle{remark}
\newtheorem{rem}{Remark}[section]

\newcommand{\UJ}{\underline{J}}
\newcommand{\lp}[1]{{\bf #1}}
\newcommand{\pe}[1]{#1^{\operatorname{>J}}}
\newcommand{\np}[1]{\mathfrak{#1}}
\newcommand{\fr}{\frac}
\newcommand{\tand}{\text{ and }}

\begin{document}
%\subjclass[2010]{35L70, 35Q55}
%\keywords{Nonlinear wave equation, nonlinear Schr\"odinger equation}

\title[Global dynamics below ground state for Zakharov]{Global dynamics below the ground state energy for the Zakharov system in the 3D radial case}
\author[Z. Guo, K. Nakanishi, S. Wang]{Zihua Guo, Kenji Nakanishi, Shuxia Wang}

\address{LMAM, School of Mathematical Sciences, Peking
University, Beijing 100871, China}
\address{Beijing International Center for Mathematical Research, Beijing
100871, China} \email{zihuaguo@math.pku.edu.cn}

\address{Department of Mathematics, Kyoto University, Kyoto 606-8502,
Japan}

\email{n-kenji@math.kyoto-u.ac.jp}

\address{LMAM, School of Mathematical Sciences, Peking
University, Beijing 100871, China}

\email{wangshuxia@pku.edu.cn}

\begin{abstract}
We consider the global dynamics below the ground state energy for
the Zakharov system in the 3D radial case.  We obtain dichotomy
between the scattering and the growup.
\end{abstract}

\maketitle

\tableofcontents

\section{Introduction}
In this paper, we continue our study \cite{GN} on the global Cauchy problem for the 3D Zakharov system \EQ{\label{eq:Zak}
 \CAS{ i\dot u - \De u = nu,\\
   \ddot n/\al^2 - \De n = -\De|u|^2,}}
with the initial data
\begin{align}
u(0,x)=u_0,\, n(0,x)=n_0,\,\dot n(0,x)=n_1,
\end{align}
where $(u,n)(t,x):\R^{1+3}\to\C\times\R$, and $\al>0$ denotes the
ion sound speed. It preserves $\|u(t)\|_{L^2_x}$ and the energy \EQ{
 E=\int_{\R^3}|\na u|^2+\frac{|D^{-1}\dot n|^2/\al^2+|n|^2}{2}-n|u|^2 dx,}
where $D:=\sqrt{-\De}$, as well as the radial symmetry.

This system \eqref{eq:Zak} in $d$ dimensions was introduced by
Zakharov \cite{Zak} as a mathematical model for the Langmuir
turbulence in unmagnetized ionized plasma. It has been extensively
studied. Local wellposedness (without symmetry) is well known. For
example, the well-posedness in the energy space was proved in
\cite{BoCo} for $d = 2, 3$ and in \cite{GTV} for $d=1$, and in
weighted Sobolev space in \cite{KPV}. It has been improved 
to the critical regularity in \cite{GTV,BHHT} for $d=1,2$, and to
the full subcritical regularity in \cite{GTV,BeHe} for $d\geq
4,d=3$. The well-posedness for the system on the torus was studied
in \cite{Takaoka,Kishi}. These results except for \cite{KPV} follow from the iteration argument using Bourgain space, where the estimates depend on $\al$, while in \cite{KPV} the well-posedness is obtained uniformly for $\al$. 
For more results on the subsonic limit to NLS (as $\alpha\to \infty$), see \cite{SW,OT,MN}. Concerning the long-time behavior, Merle \cite{Merle} obtained blow-up in finite or infinite time for negative energy (which we will call grow-up for brevity), while the scattering theory was 
studied in \cite{Shimo,GV,OT2}, dealing with solutions for given asymptotic free profiles. 
Recently, in \cite{GN} the authors obtained scattering for radial initial data with small energy in the 3D case, by using the normal form reduction
and radial-improved Strichartz estimates. The purpose of this paper
is to consider the global dynamics for larger data under the
radial symmetry.

To simplify the presentation, we rewrite the system into the first order as
usual. Let $N:=n-iD^{-1}\dot n/\al$. Then \eqref{eq:Zak} can be
rewritten as \EQ{\label{eq:Zak2}
 \pt(i\p_t-\De) u = (\Re N)u,
 \pq(i\p_t+\al D)N = \al D|u|^2,}
with initial data $(u_0,N_0)\in H^1\times L^2$. It has the conserved
mass 
\EQ{
 M(u):=\int_{\R^3}\frac{|u|^2}{2}dx,}
and the Hamiltonian 
\EQ{
 E_Z(u,N):=\int_{\R^3}\frac{|\na u|^2}{2}+\frac{|N|^2}{4}-\frac{\Re N|u|^2}{2}dx
 = E_S(u)+\|N-|u|^2\|_{L^2}^2/4,}
where $E_S(u)$ denotes the Hamiltonian for the cubic NLS (the limit
$\al\to\I$)
\begin{align}\label{eq:cNLS}
(i\p_t-\De) u = |u|^2u,
\end{align}
namely
\EQ{
 E_S(u):=\int_{\R^3}\frac{|\na u|^2}{2}-\frac{|u|^4}{4}dx.}

Let $Q$ be the ground state for NLS \eqref{eq:cNLS}, that is the
unique positive radial solution for the following equation \EQ{
\label{stat-KG}
 -\De Q + Q = Q^3,}
which minimizes the action \EQ{
 J(Q):=E_S(Q)+M(Q)}
among all nontrivial solutions of \eqref{stat-KG} (see, e.g., \cite{HR} for
further properties of $Q$). For $\lambda>0$, let \EQ{
 Q_\la(x):=\la Q(\la x),}
then we have \EQ{
 -\De Q_\la + \la^2 Q_\la = Q_\la^3, \pq M(Q_\la)=\la^{-1}M(Q),\pq E_S(Q_\la)=\la E_S(Q).}
Thus the Zakharov system \eqref{eq:Zak2} has the following family of
radial standing waves \EQ{\label{eq:zakstandingwave}
 (u,N)=(e^{i(\te-\la^2 t)}Q_\la,Q_\la^2),}
where $\la>0$ and $\te\in\R$ can be chosen arbitrarily.

The goal of this study is to determine global dynamics of all the
radial solutions ``below" the above family of special solutions, in
the spirit of Kenig-Merle \cite{KM}, namely the variational
dichotomy into the scattering solutions and the blowup solutions.
Such a result has been obtained for the limit equation
\eqref{eq:cNLS} by Holmer-Roudenko \cite{HR} in the radial case, as
well as in the nonradial case \cite{DHR}. For the dichotomy, we need
to introduce another functional (for NLS), which is the scaling
derivative of the action $J$: \EQ{
 K(\fy):=\p_{\la}|_{\la=1}J(\la^{d/2}\fy(\la x))
 =\int_{\R^d}|\na\fy|^2-\frac{d|\fy|^4}{4}dx.}
We would like to get the same result as in \cite{HR} for NLS, 
but by the virial argument as in \cite{Merle} we can only prove grow-up, due to the poor control of the wave component $N$. 
In fact,
existence of any blowup in finite time is still an open question for
the 3D Zakharov system. The main result of this paper is

\begin{thm}\label{thm}
Assume that $(u_0,N_0)\in H^1(\R^3)\times L^2(\R^3)$ is radial and
satisfies
\begin{align}\label{eq:Ethm}
  E_Z(u_0,N_0)M(u_0)<E_S(Q)M(Q).
\end{align}
Then we have

(a) if $K(u_0)\ge 0$, then \eqref{eq:Zak2} has a unique global solution $(u,N)$, which scatters both as $t\to\I$ and as $t\to-\I$
in the energy space. More precisely, there are $(u_\pm,N_\pm)\in
H^1\times L^2$ such that \EQ{
 \|(u(t),N(t))-(e^{-it\De}u_\pm,e^{it\al D}N_\pm)\|_{H^1\times L^2}\to 0 \pq (t\to\pm\I).}

(b) if $K(u_0)<0$, then \eqref{eq:Zak2} blows up in either finite or infinite time, in the sense that $\sup_{0<t<T^*}\norm{(u,N)}_{H^1(\R^3)\times L^2(\R^3)}=\I=\sup_{T_*<t<0}\norm{(u,N)}_{H^1(\R^3)\times L^2(\R^3)}$, where $(T_*,T^*)$ is the maximal interval of existence. 
\end{thm}

\begin{rem}
1) Assuming $K(u_0)=0$ and \eqref{eq:Ethm}, one can actually get
by variational estimates that $u_0=0$, so $u\equiv 0$ and $N=e^{it\al D}N_0$, see Section 2.

2) The condition \eqref{eq:Ethm} is sharp in view of the standing
wave solutions \eqref{eq:zakstandingwave}.

\end{rem}

The difficulty for the scattering even for small data can be observed by comparing the time decay with the NLS of general power nonlinearity 
\EQ{ 
 i\dot u - \De u = |u|^pu, \pq u:\R^{1+d}\to\C.}
It is well known that the scattering for NLS requires $p>2/d$, corresponding to the time integrability of the optimal decay of to the potential 
\EQ{
 \||u|^p\|_{L^\I_x}\sim |t|^{-dp/2},} 
while the scattering in $H^s$ for any $s$ has been proven only for $p\ge 4/d$. The 3D Zakharov system would be on the boarderline in the above sense, since the potential $n$ can decay only by 
\EQ{
 \|n\|_{L^\I_x} \sim |t|^{-1},}
as it is solving the 3D wave equation. This suggests that the decay estimates are far insufficient for the scattering in $H^1$, and so it is essential to exploit nonlinear oscillations, e.g.~by the normal form. This part for small radial data has been resolved in the previous paper \cite{GN}. Hence our main task in this paper is to carry out the Kenig-Merle approach \cite{KM} in accordance with the normal form. Since the normal form produces nonlinear terms without time integration, we need to modify Kenig-Merle's formulation, as well as some estimates in \cite{GN}. As a crucial ingredient for that approach, we will derive a virial identity, which is slightly different from Merle's one in \cite{Merle} and more suitable for the scattering.

\section{Hamiltonian and variational structures}

\subsection{Virial identity} We derive a virial identity on $\R^d$, which is slightly different from \cite{Merle}. Recall that the Zakharov system can be rewritten in the
Hamiltonian form \EQ{
 \p_t\mat{u \\ N} = \mathbf{J}E_Z'(u,N),}
where $\mathbf{J}$ and $E_Z'$ denote the symplectic operator and the
Frech\'et derivative given by
\begin{align*}
 \mathbf{J} = \mat{i & 0 \\ 0 & 2i\al D}, \pq E_Z'(u,N)=\mat{E_S'(u)-(\Re N-|u|^2)u \\ (N-|u|^2)/2}=\mat{-(\Re N)u \\ (N-|u|^2)/2}.
\end{align*}
Let $A$ be the generator for the family of scaling
transforms\footnote{The order of scaling, i.e.~the exponents $\frac d2$
and $\frac{d+1}2$, is the unique choice such that \eqref{J commute} holds.} 
\EQ{
 \mat{f \\ g} \mapsto S_\la\mat{f \\ g}:=\mat{\la^{d/2}f(\la x) \\ \la^{(d+1)/2}g(\la x)} \pq(\la>0),}
hence we have \EQ{
 \pt A=\mat{x\cdot\na +d/2 & 0 \\ 0 & x\cdot\na+(d+1)/2},
 \pr A^*=-\mat{x\cdot\na +d/2 & 0 \\ 0 & x\cdot\na + (d-1)/2}.}
Let $w:=(u,N)$, $\mathbf{\UJ}:=\mathbf{J}^{-1}$ and denote the real
part of $L^2$ inner product by $\LR{\cdot|\cdot}$. Then the virial
identity for the Zakharov system is given by \EQ{ \label{virial id}
 \pt\p_t\LR{\mathbf{\UJ} v|Av}=\LR{\mathbf{\UJ} \dot v|Av}+\LR{\mathbf{\UJ} v|A\dot v}
 =\LR{\dot v|(\mathbf{\UJ}^*A+A^*\mathbf{\UJ})v}
 \pr=2\LR{\dot v|\mathbf{\UJ}^*Av}=2\LR{JE_Z'(v)|\mathbf{\UJ}^*Av}=2\LR{E_Z'(v)|Av}
 \pr=2\p_{\la=1}E_Z(S_\la v)=2\p_{\la=1}[E_S(S_\la u)+\|\la^{1/2}N-\la^{d/2}|u|^2\|_2^2/4]
 \pr=2K(u)+\frac{1}{2}\|N-|u|^2\|_2^2 - \frac{d-1}{2}\LR{N-|u|^2||u|^2},}
where we used for the third equality that 
\EQ{ \label{J commute}
 \mathbf{J}^*A^* = i\mat{x\cdot\na+d/2 & 0 \\ 0 & 2\al(x\cdot\na+(d+1)/2)D}=A\mathbf{J}.}
Therefore, we have proved
\begin{lem}[Virial identity]\label{lem:virial}
Assume $v=(u,N)$ is a smooth decaying solution to
Zakharov system \eqref{eq:Zak}. Then
\EQ{ \label{our vir}
 \p_t\LR{\mathbf{\UJ} v|Av}\pt=\p_t\bigl[\LR{u|ir\p_r u}+\frac{1}{2\al}\LR{N|ir\p_r D^{-1}N}\bigr]
 \pr=2K(u)+\frac{1}{2}\|N-|u|^2\|_2^2 -
\frac{d-1}{2}\LR{N-|u|^2||u|^2}. }
\end{lem}

The virial identity by Merle \cite{Merle} is slightly different from the above one. In our notation, it can be written as 
\EQ{ \label{Merle vir}
 \p_t\bigl[\LR{u|ir\p_r u}\pt-\frac{1}{\al}\LR{\Re N|r\p_r D^{-1}\Im N}\bigr]
 \pr=2K(u)+\frac{d}{2}\|N-|u|\|_2^2-(d-1)\|\Im N\|_2^2
 \pr=2dE_Z(u,N)-(d-2)\|\na u\|_2^2-(d-1)\|\Im N\|_2^2.}
The left hand side differs from \eqref{our vir} since $ir\p_r D^{-1}$ is not self-adjoint, but $ir(\p_r+(d-1)/2)D^{-1}$ is so. Precisely, the difference is 
\EQ{
 \p_t\frac{d-1}{2\al}\LR{\Re N|D^{-1}\Im N}=\frac{1-d}{2}\LR{N-|u|^2|N}+(d-1)\|\Im N\|_2^2.} 
The advantage of our identity is that it is monotone both in the scattering region ($K>0$) and in the blow-up region ($K<0$), as we will show in the next section, while \eqref{Merle vir} is not monotone when $u(t)$ and $n(t)$ are very small compared with $\dot n(t)$. 
Although Merle's identity is more convenient in the blow-up region, our identity can also be used there, as we will see in Section \ref{sec:growup}. 

\subsection{Variational estimates}
In the 3D case $d=3$, the cubic nonlinearity is $L^2$-supercritical
and $\dot H^1$ subcritical. Hence $Q$ is obtained by the constrained
minimization \EQ{\label{eq:Qconst}
 J(Q)=\inf\{J(\fy)\mid 0\not=\fy,\ K(\fy)=0\}.}
Indeed, $Q$ is the unique minimizer modulo the phase $e^{i\te}$ and
spatial tranlation. By scaling, we also have for any $\la>0$\EQ{
 \la J(Q)=J_\la(Q_\la)=\inf\{J_\la(\fy)\mid 0\not=\fy,\ K(\fy)=0\}, \pq J_\la:=E_S+\la^2 M,}
and $Q_\la$ is the unique minimizer modulo phase and translation.

\begin{lem}\label{lem:K2sgn}
Assume that $(u,N)$ is a solution  to \eqref{eq:Zak2} with maximal
interval $I$ satisfying \EQ{E_Z(u,N)M(u)<E_S(Q)M(Q).} Then
for some $\la>0$ we have $ E_Z(u,N)+\la^2M(u)<\la J(Q)$. 
Moreover, either $u\equiv 0$ on $I$, or $K(u(t))\ne 0$ for all $t\in I$. In other words, $K(u(t))$ does not change its sign on $I$. 
\end{lem}
\begin{proof}
From \eqref{eq:Qconst}, we have $J(Q)=\inf_{\la>0}J(Q_\la)$, and
thus $\partial_\la |_{\la=1} J(Q_\la)=0$. This implies
\[J^2(Q)/4 = E_S(Q)M(Q).\]
Thus we see that there exists $\la>0$ such that \EQ{
 E_Z(u,N)+\la^2M(u)<J_\la(Q_\la)=\la J(Q).}

Since $J_\la(u)\le E_Z(u,N)+\la^2M(u)$, 
by the variational characterization of $Q_\la$, we have at each $t\in I$, 
\EQ{
 K(u(t))=0 \iff u(t)=0.}
If $K(u(t_0))=0$ for some $t_0\in I$, by uniqueness
we have $u\equiv 0$. 
\end{proof}

\begin{cor}\label{cor:cor}
Assume that $(u,N)$ is a solution to \eqref{eq:Zak2} with maximal
interval $I$ satisfying for some $\la>0$ \EQ{\label{eq:Econst}
E_Z(u,N)+\la^2M(u)<\la J(Q),\ K(u_0)\geq 0.} Then $I=(-\I,\I)$, and 
moreover,
\begin{align}\label{eq:Eequiv}
E_Z(u,N)+\la^2M(u)\sim \norm{u}_{H^1}^2+\norm{N}_{L^2}^2\sim
\norm{u_0}_{H^1}^2+\norm{N_0}_{L^2}^2.
\end{align}
where the implicit constant depends only on $\la$ and $J(Q)$.
\end{cor}

\begin{proof}
From Lemma \ref{lem:K2sgn} (b) we get that if $K(u_0)=0$, then
$u\equiv 0$, and hence this case is trivial. Thus we may assume
$K(u_0)>0$, hence $K(u(t))>0$ by Lemma \ref{lem:K2sgn} (b). From
the assumption, we get \eqref{eq:Eequiv} immediately from
\begin{align*}
\la J(Q)\geq& E_Z(u,N)+\la^2M(u)-K(u(t))/3\\
=&\frac{1}{6}\|\na u\|_2^2+\fr {\la^2}2\|u\|_2^2+\frac
14\|N-|u|^2\|_2^2,
\end{align*}
and the Sobolev inequality $\norm{u}_{L^4}\les \norm{u}_{H^1}$. So $(u,N)(t)$ is a priori bounded in $H^1\times
L^2$, and thus by the local wellposedness we have $I=(-\I,\I)$.
\end{proof}

So far, the global well-posedness of part (a) of Theorem \ref{thm} is
proved. It remains to prove the scattering and part (b). For both
purposes, the virial estimates play crucial roles. Unlike the NLS
case, it is not at all obvious that virial for \eqref{eq:Zak2} is
monotone. The following lemma is our key observation

\begin{lem}\label{lem:var}
Let $\fy\in H^1(\R^3)$, $\la>0$ and $\nu\ge 0$ satisfy \EQ{
 E_S(\fy)+\la^2M(\fy)+\frac{\nu^2}{4} \le J_\la(Q_\la).}
Then we have \EQ{
 \CAS{K(\fy) \ge 0 \implies 4K(\fy) + \nu^2 \ge \sqrt{6}\nu\|\fy\|_4^2,\\
 K(\fy)\le 0 \implies 4K(\fy) + \nu^2 \le -2\nu\|\fy\|_4^2.}}
\end{lem}
\begin{proof}
First, if $K(\fy)=0$ then $\nu=0$ and the conclusion is trivial.
Hence we may assume $K(\fy)\not=0$ as well as $\nu>0$. Next by the
scaling $(\fy,\nu)\mapsto(\la\fy(\la x),\sqrt{\la}\nu)$, we may
remove $\la$ or assume $\la=1$. Then the energy constraint becomes
$J(\fy)+\nu^2/4\le J(Q)$. Now consider the $L^2$ scaling of $\fy$, $S_\mu\fy=\mu^{d/2}\fy(\mu x)$
and \EQ{
 \pt J(S_\mu\fy)=\frac{\mu^2}{2}\|\na\fy\|_2^2+\frac{1}{2}\|\fy\|_2^2-\frac{\mu^3}{4}\|\fy\|_4^4,
 \pr \mu\p_\mu J(S_\mu\fy)=K(S_\mu\fy)=\mu^2\|\na\fy\|_2^2-\frac{3\mu^3}{4}\|\fy\|_4^4.}
There is a unique $0<\mu\not=1$ such that \EQ{ \label{eq mu}
 \|\na\fy\|_2^2=\frac{3\mu}{4}\|\fy\|_4^4,}
which is equivalent to $K(S_\mu\fy)=0$. Then the variational
characterization of $Q$ implies $J(S_\mu\fy)\ge J(Q)$, and so \EQ{
 \frac{\nu^2}{4} \le J(S_\mu\fy)-J(\fy)
  \pt=\frac{\mu^2-1}{2}\|\na\fy\|_2^2-\frac{\mu^3-1}{4}\|\fy\|_4^4,
 \pr=\frac{(\mu-1)^2(\mu+2)}{8}\|\fy\|_4^4,}
where \eqref{eq mu} is used in the last step. Let
$X:=\|\fy\|_4^2/\nu$. Then the above inequality is rewritten as \EQ{
\label{X lowbd}
 |\mu-1|\sqrt{\mu+2}X \ge \sqrt{2}.}
Hence it suffices to estimate, under the above constraint, \EQ{
 \frac{4K(\fy)+\nu^2}{\nu\|\fy\|_4^2}=3(\mu-1)X+1/X=:f(X,\mu).}
For $K(\fy)>0$, or equivalently $\mu>1$, $f(X,\mu)$ is increasing
in $X$ unless \EQ{
 \sqrt{\frac{1}{3(\mu-1)}}<\frac{1}{\mu-1}\sqrt{\frac{2}{\mu+2}},}
which is solved $\mu>(\sqrt{33}-1)/2$. In the latter case, we have
\EQ{
 3(\mu-1)X+1/X \ge 2\sqrt{3(\mu-1)X/X}>\sqrt{6},}
since $\mu>3/2$. Otherwise, the minimum is attained at the boundary
and equal to \EQ{
 f(\frac{1}{\mu-1}\sqrt{\frac{2}{\mu+2}},\mu)=3\sqrt{\frac{2}{\mu+2}}+(\mu-1)\sqrt{\frac{\mu+2}{2}}=:b(\mu),}
which is increasing\footnote{This can be checked by computing
$\frac{d(b^2)}{d\mu}$.} in $\mu>0$, hence $b(\mu)>b(1)=\sqrt{6}$.

For $K(\fy)<0$, or equivalently $0<\mu<1$, $-f(X,\mu)$ is
increasing in $X$, so its minimum is attained at the boundary and
equals to \EQ{
 -f(\frac{1}{1-\mu}\sqrt{\frac{2}{\mu+2}},\mu)=b(\mu)>b(0)=2.}
Therefore, the proof of the lemma is completed.
\end{proof}

\begin{rem} Applying the lemma above by letting \EQ{
 \nu:=\|N-|u|^2\|_2,}
we get from Lemma \ref{lem:virial} that the virial $\LR{\mathbf{\UJ}
v|Av}$ is monotone in our consideration. This fact will play crucial
role in our consequent analysis.
\end{rem}

\section{Growup at infinity}\label{sec:growup}

This section is devoted to prove part (b) of Theorem \ref{thm}. We
assume that under the assumption of part (b), the solution exists
for all $t>0$. We will show that $\sup\limits_{t>0}\norm{(u,N)}_{H^1(\R^3)\times L^2(\R^3)}=\I$.

\subsection{Localized virial} Let $X=X^*$ be the operator of smooth
trancation to $|x|<R$ by multiplication with $\psi_R(x)=\psi(x/R)$, where $\psi\in C_0^\I(\R^3)$ is a fixed radial function satisfying $0\le\psi\le 1$, $\p_r\psi\le 0$, $\psi(x)=1$ for $|x|\le 1$ and $\psi(x)=0$ for $|x|\ge 2$. 
We consider the localized virial quantity in the form \EQ{
 V_R(t):=\LR{\UJ v|(AX+XA)v}.}
Then similarly to the non-localized virial identity, we can compute
\EQ{ \label{local virial}
 \dot V_R=\LR{E_Z'(v)|(AX+XA+AJX\UJ+J^*X\UJ^*A)v}.}
Putting $\nu:=n-|u|^2$, the right hand side can be written
componentwise \EQ{
 \dot V_R=\pt\LR{E_S'(u)-\nu u|2A_0Xu+2XA_0u}
 \pr+\LR{\nu/2|(XA_1+A_1X+DXD^{-1}A_1+A_1DXD^{-1})(\nu+|u|^2)},}
where $A_j:=x\cdot\na+(d+j)/2$. The right hand side is decomposed
into the NLS part: \EQ{
 NS:=\LR{E_S'(u)|2A_0Xu+2XA_0u},}
the quadratic terms in $\nu$: \EQ{
 QN:=\LR{\nu/2|(XA_1+A_1X+DXD^{-1}A_1+A_1DXD^{-1})\nu},}
and the cubic cross terms: \EQ{
 CC:=\pt\LR{-\nu u|2A_0Xu+2XA_0u}
 \pr+\LR{\nu/2|(XA_1+A_1X+DXD^{-1}A_1+A_1DXD^{-1})|u|^2},}
i.e.,~$\dot V_R=NS+QN+CC$. Since the NLS part has been treated by
Ogawa-Tsutsumi \cite{OgT} and Holmer-Roudenko \cite{HR}, while the cross terms are higher
order, the main problem for us is to control $QN$. Indeed, our way of the localization is motivated by a better cancellation in $QN$, while some other multipliers such as $AXv$ in \eqref{local virial} could make the other terms simpler. 

It is further
decomposed $QN=(QN_1+QN_2+QN_3)/2$ with \EQ{
 QN_1:=\LR{\nu|(XA_1+A_1X)\nu}=\LR{\nu|X(A_1+A_1^*)\nu}=\LR{\nu|X\nu},}
where we used the symmetry of the bilinear form as well as $X=X^*$
and $A_0=-A_0^*$. Putting $\y:=D^{-1}\nu$, the other two terms are
computed as follows. \EQ{
 QN_2:\pt=\LR{\nu|DXD^{-1}A_1\nu}=\LR{\y|D^2XA_{-1}\y}=\LR{\na\y|\na XA_{-1}\y}
 \pr=\LR{\na\y|XA_1\na\y}+\LR{\na\y|(\na\psi_R)A_{-1}\y},}
where we used $DA_{-1}=A_{1}D$ and $\na A_{-1}=A_{1}\na$, \EQ{
 QN_3:=\pt=\LR{\nu|A_1DXD^{-1}\nu}=\LR{\y|DA_1DX\y}=\LR{\na\y|A_1\na X\y}
 \pr=\LR{\na\y|A_1X\na\y}+\LR{\na\y|A_1(\na\psi_R)\y},}
where we used $DA_1D=-\na\cdot A_1\na$. Hence \EQ{
 QN_2+QN_3\pt=\LR{\na\y|X(A_1+A_1^*)\na\y}+\LR{\na\y|(\na\psi_R)A_{-1}\y+A_1(\na\psi_R)\y}
 \pr=\LR{\na\y|X\na\y}+2\LR{\na\y|(\na\psi_R)x\cdot\na\y}+\LR{\na\y|\y A_{d}\na\psi_R}
 \pr=\LR{\na\y|X\na\y}+2\LR{\y_r|r\psi_R'\y_r}-\frac12\LR{|\y|^2|A_{d-2}\De\psi_R},}
where we used the radial symmetry of $\psi_R$ but not of $\y$. Thus
we obtain \EQ{
 QN=\LR{\nu|X\nu}/2 + \LR{\na\y|X\na\y}/2 + \LR{\y_r|r\psi_R'\y_r}-\frac14\LR{|\y|^2|A_{d-2}\De\psi_R}.}
The first two terms are less than $\|\nu\|_2^2=\|\na\y\|_2^2$ since
$\psi_R\le 1$, while the third term is nonpositive since $\psi_R'\le
0$. The last term is bounded from above and below by\footnote{Such an error term does not occur in Merle's virial identity \cite{Merle}. This is a disadvantage of our identity. Nevertheless we can dispose of it using the evolution equation.} 
\EQ{
 \ro_R:=\int_{|x|\sim R}\frac{|\y|^2}{R^2}dx \lec \|\na\y\|_2^2=\|\nu\|_2^2.}
In short, we have \EQ{
 QN(t) \le \|\nu\|_2^2 + O(\ro_R(t)).}
$\ro_R(t)\to 0$ as $R\to\I$ for each fixed $t$, but some uniform
decay is needed for the main term $\dot
V_\I(t)=4K+\|\nu\|_2^2+(1-d)\LR{\nu||u|^2}$ to absorb the error.
For that we use the equation of $\y$: \EQ{
 (i\p_t+\al D)\y=D^{-1}(i\p_t+\al D)(N-|u|^2)=-iD^{-1}|u|^2_t,}
and the corresponding integral equation \EQ{
 \y\pt=\y^0+\y^1,\pq \y^0:=e^{i\al Dt}\y(0),
 \\\y^1&:=-i\int_0^t\frac{\sin(\al D(t-s))}{D^2}|u(s)|^2_sds
 \pr=-i\frac{\sin(\al Dt)}{D^2}|u(0)|^2-i\al\left[\int_0^{(t-1)_+}+\int_{(t-1)_+}^t\right]\frac{\cos(\al D(t-s))}{D}|u(s)|^2ds
 \pr=:\y^2+\y^3+\y^4.}
We use the above equation only for very low frequency. More
precisely, with a small parameter $0<\de<1$ independent of $t$,
decompose $\y$ smoothly in the Fourier space \EQ{
 \y=\y_{<\de}+\y_{>\de}, \pq \y_{<\de}:=\F^{-1}\psi_\de\F\y,}
then we have $\|\y_{>\de}\|_2\le\de^{-1}\|\nu\|_2$. For the low
frequency part, we have \EQ{
 \pt \|\y^0_{<\de}\|_{\dot H^1}=\|\nu_{<\de}(0)\|_{L^2},
 \pr \|\y^2_{<\de}\|_{\dot H^{1/2+}}\lec \||u(0)|^2\|_{L^1}\lec \|u(0)\|_2^2,
 \pr \|\y^4_{<\de}\|_{\dot H^{-1/2+}}\lec \al\||u|^2\|_{L^\I_tL^1_x}\lec \al\|u(0)\|_2^2,
 \pr \|\y^3_{<\de}\|_{L^{4+}}\lec \|\y^3_{<\de}\|_{\dot B^{-1}_{\I,\I}}^{1/2}\|\y^3_{<\de}\|_{\dot H^1}^{1/2},}
and by the $L^\I$ decay of the wave equation, \EQ{
 \|\y^3_{<\de}\|_{\dot B^{-1}_{\I,\I}}
 \lec \int_0^{(t-1)_+}\frac{1}{|t-s|}\|u(s)\|_2^2ds \lec \|u(0)\|_2^2\log(t+1).}
Thus we obtain \EQ{
 \|\y^1_{<\de}\|_{L^\I_t(0,T;L^{4+}_x)} \lec \|u(0)\|_2^2\log(T+2)+\|\nu\|_{L^\I_t(0,T;L^2_x)}^2,}
and so \EQ{
 \sup_{0<t<T}\ro_R(t) \lec \pt\|\nu_{<\de}(0)\|_2^2+R^{-1/2+}[\|u(0)\|_2^2\log(T+2)+\|\nu\|_{L^\I_t(0,T;L^2_x)}^2]
 \pr+(\de R)^{-2}\|\nu\|_{L^\I_t(0,T;L^2_x)}^2.}

Next we estimate the cubic cross terms \EQ{
 \pt CC=CC_1+CC_2+CC_3,
 \pr CC_1:=-2\LR{\nu u|A_0Xu+XA_0u}=-2\LR{\nu|(r\psi_R'+XA_d)|u|^2},
 \pr CC_2:=\frac12\LR{\nu|(XA_1+A_1X)|u|^2}=\LR{\nu|(XA_1+r\psi_R'/2)|u|^2},
 \pr CC_3:=\frac12\LR{\nu|(DXD^{-1}A_1+A_1DXD^{-1})|u|^2}.}
For the last term we use the commuting relations: \EQ{
 A_1DXD^{-1}=DA_{-1}XD^{-1}\pt=D(XA_{-1}+r\psi_R')D^{-1}
 \pr=DXD^{-1}A_1+Dr\psi_R'D^{-1},}
and so \EQ{
 \pt CC_3=\LR{\nu|XA_1|u|^2}+CC_3',
 \pr CC_3':=\LR{\nu|([D,X]D^{-1}A_1+[D,r\psi_R']D^{-1}/2)|u|^2}.}
Hence \EQ{
 CC=(1-d)\LR{\nu||u|^2}+\LR{\nu|[(1-d)(\psi_R-1)-3r\psi_R'/2]|u|^2}+ CC_3',}
and the second term on the right is bounded by \EQ{
 \int_{|x|\gec R}|\nu u^2|dx \pt\lec \|\nu\|_2\|u\|_2\|u\|_{L^\I(|x|\gec R)}
 \pn\lec R^{-1}\|\nu\|_2\|u\|_2^{3/2}\|\na u\|_2^{1/2},}
since the functions in the brackets $[]$ vanish on $|x|\lec R$. We used the
radial Sobolev inequality \EQ{
 \fy(x)=\fy(|x|)\in H^1(\R^3)\implies \|r\fy\|_{L^\I(\R^3)} \lec \|\fy\|_2^{1/2}\|\na\fy\|_2^{1/2}.}
For the commutator terms $CC_3'$, we use the elementary commutator
estimate \EQ{
 \|[D,f]g\|_{L^2} \lec \|\F(\na f)\|_{L^1}\|g\|_{L^2},}
together with the (radial/nonradial) Sobolev \EQ{
 \pt \|xu^2\|_2 \le \|xu\|_\I \|u\|_2 \lec \|u\|_2^{3/2}\|\na u\|_2^{1/2},
 \pr \|D^{-1}|u|^2\|_2 \lec \||u|^2\|_{6/5} \le \|u\|_2\|u\|_3 \lec \|u\|_2^{3/2}\|\na u\|_2^{1/2}.}
Since $\|\F(\na\psi_R)\|_1=CR^{-1}$, we thus obtain \EQ{
 |CC_3'| \lec \|\nu\|_2R^{-1}[\|D^{-1}\na\cdot x|u|^2\|_2+\|D^{-1}|u|^2\|_2]
 \lec R^{-1}\|\nu\|_2\|u\|_2^{3/2}\|\na u\|_2^{1/2}.}
In short, we have obtained \EQ{
 CC = (1-d)\LR{\nu||u|^2}+O(R^{-1}\|\nu\|_2\|u\|_2^{3/2}\|\na u\|_2^{1/2}).}

Finally we estimate the NLS part \EQ{
 NS/2 \pt=\LR{-\De u-|u|^2u|A_0Xu+XA_0u}
 \pr=\LR{\na u|\na(A_0X+XA_0)u}-\LR{r\psi_R'||u|^4}-\LR{\psi_R|(r\p_r/2+d)|u|^4}\pr=:NS_1+NS_2+NS_3.}
For the first term $NS_1$ we use \EQ{
 \na(A_0X+XA_0)\pt=A_2\na X+X\na A_0+[\na\psi_R]A_0
 \pr=A_2X\na+A_2[\na\psi_R]+XA_2\na+[\na\psi_R]A_0
 \pr=(A_2X+XA_2)\na+2[\na\psi_R]r\p_r+[A_{2+d}\na\psi_R],}
where the bracket denotes the multiplication with the inside
function. Using $A_0^*=-A_0$ as well, we obtain \EQ{
 NS_1=\LR{\na u|2X\na u}+2\LR{u_r|\psi_R'ru_r}+\frac12\LR{\na|u|^2|A_{2+d}\na\psi_R}.}
Since $\psi_R\le 1$ and $\psi_R'\le 0$, the first term is less than
$2\|\na u\|_2^2$ and the second is nonpositive. The last term equals
\EQ{
 -\frac 12\LR{|u|^2|A_d\De\psi_R} \lec\|u\|_2^2\|A_d\De\psi_R\|_\I\lec R^{-2}\|u\|_2^2.}
The quartic terms equal \EQ{
 NS_2+NS_3=-\frac 12\LR{(r\p_r+d)\psi_R||u|^4}=-\frac{d}{2}\|u\|_4^4-\frac 12\LR{r\psi_R'+d(\psi_R-1)||u|^4},}
and the last term is bounded by \EQ{
 \|u\|_{L^4(|x|\gec R)}^4 \le \|u\|_2^2\|u\|_{L^\I(|x|\gec R)}^2
 \lec R^{-2}\|u\|_2^3\|\na u\|_2,}
using the radial Sobolev inequality. In short, we have obtained \EQ{
 NS/2 \le 2K(u) + O(R^{-2}\|u\|_2^3\|\na u\|_2).}

Gathering the above estimates on $QN$, $CC$ and $NS$, we obtain \EQ{
 \dot V_R \pt\le 4K(u) + \|\nu\|_2^2 + (1-d)\LR{\nu||u|^2} + O(\ro_R)
 \pr\pq+ O(R^{-1}\|\nu\|_2\|u\|_2^{3/2}\|\na u\|_2^{1/2}) + O(R^{-2}\|u\|_2^3\|\na u\|_2),}
and \EQ{
 \sup_{0<t<T}\ro_R \lec \pt\|\nu_{<\de}(0)\|_2^2+R^{-1/2+}[\|u(0)\|_2^2\log(T+2)+\|\nu\|_{L^\I_t(0,T;L^2_x)}^2]
 \pr+(\de R)^{-2}\|\nu\|_{L^\I_t(0,T;L^2_x)}^2.}
Also we have \EQ{
 |V_R| \lec R[\|u\|_2\|\na u\|_2+\|N\|_2^2].}

Now suppose for contradiction that \EQ{
 \sup_{t>0}\|u(t)\|_{H^1_x}+\|N(t)\|_{L^2_x}\le M\in[1,\I),}
then $\|\nu\|_{L^2_x}\lec M^2$ and $|V_R|\lec RM^2$. The variational
estimate provides us with an upper bound \EQ{
 \dot V_\I= 4K(u) + \|\nu\|_2^2 + (1-d)\LR{\nu||u|^2} \le -\ka}
for some $\ka\sim J_\la(Q_\la)-[E_Z(v)+\la^2M(u)]>0$. We can first
choose $0<\de\ll 1$ so small that $\|\nu_{<\de}(0)\|_2^2\ll\ka$.
Secondly we can choose $R\gg 1$ so large that \EQ{
 \pt R^{-1/3}M^2\log(RM^2/(\de\ka)) \ll \ka,
 \pq (R^{-1/3}+(\de R)^{-2})M^4 \ll \ka,}
where $\log(RM^2/(\de\ka))$ may be replaced with
$(RM^2/(\de\ka))^{1/6}$ for example. Then for $0<t<RM^2/\de\ka=:T$,
we have $\dot V_R\le -\ka/2$, and so $|V_R(T)-V_R(0)|\ge \ka
T/2=\frac{RM^2}{2\de}$, which is contradicting the above bound on
$|V_R|$.

\section{Concentration-compactness procedure}

It remains to prove the scattering in part (a) of Theorem \ref{thm}.
Thanks to the variational estimates in Section 2, we can proceed as
Kenig-Merle. For each $0\le a\le J(Q)$ and $\la>0$, let \EQ{
\pt\mathscr{E}_\la (f,g):=\la^{-1}E_Z(f,g)+\la M(f),
 \pr \K^+_\la(a):=\{(f,g)\in H^1_r\times L^2_r\mid \mathscr{E}_\la (f,g)<a,\ K(f)\ge0\},
 \pr \cS_\la(a):=\sup\{\|(u,N)\|_S \mid (u(0),N(0))\in\K^+_\la(a),\ \text{$(u,N)$ sol.}\},}
where $S$ denotes a norm containing almost all the Strichartz norms for radial free solutions, including $L^\I_t(H^1\times L^2)$. See \eqref{def S} for the precise definition. For any time interval $I$, we will denote by $S(I)$ the restriction of $S$ onto $I$. 

From Corollary \ref{cor:cor} we already know that all solutions
starting from $\K^+_\la(a)$ stays there globally in time. What we
want to prove is the uniform scattering below the ground state
energy, i.e. $\cS_\la(a)<\I$ for all $a<J(Q)$.
Let \EQ{
 E_\la^* :=\sup\{a>0 \mid \cS_\la(a)<\I\}.}
The small data scattering in \cite{GN} implies that $E_\la^*>0$, and
the existence of the ground state soliton implies that $E_\la^*\le
J(Q)$. We will prove  $E_\la^*= J(Q)$ by contradiction, and thus
finish the proof of Theorem \ref{thm} (a). The main result in this
section is
\begin{lem}[Existence of critical element]\label{lem:crit}
Suppose $E_\la^*< J(Q)$, then there is a global solution $(u,N)$ in
$\K^+_\la(a)$ satisfying
 \EQ{
 \mathscr{E}_\la (u,N)=E_\la^*,\ \ \ \|(u,N)\|_{S(-\I,0)}=\|(u,N)\|_{S(0,\I)}=\I.
 }
 Moreover, $\{(u,N)(t)\ |\ t\in\mathbb{R}\}$ is precompact in $H^1_x\times L_x^2$.
\end{lem}

We will prove this lemma by following the concentration-compactness
procedure. The main difference from NLS is that we need to work with the solutions after the normal form transform. In particular, we have some nonlinear terms without time integration (or the Duhamel form). 
Besides that, we have various different interactions, for which we need to use
different norms or exponents.

\subsection{Profiles for the radial Zakharov}
First we recall the free profile decomposition of Bahouri-G\'erard
type \cite{BG}. Actually we do not need its full power, as we can freeze scaling and space positions of the profiles thanks to the radial symmetry and the regularity room of our problem. Hence the setting is essentially the same as the NLS case \cite{HR}. 
\begin{lem} \label{free prof}
For any bounded sequence $(f_n,g_n)$ in $H^1_r\times L^2_r$, there
is a subsequence $(f_n',g_n')$, $\bar{J}\in \N\cup \{\infty\}$, a
bounded sequence $\{\lp f^j,\lp g^j\}_{1\le j< \bar{J}}$ in
$H^1_r\times L^2_r$, and sequences $\{t_n^j\}_{n\in\N,1\le j<
\bar{J}}\subset\R$, such that the following holds. For any $0\leq
j\leq J<\bar{J}$, let \EQ{
 \pt u_n(t):=e^{-it\De}f_n', \pq N_n(t):=e^{it\al D}g_n',
 \pr \lp u_n^j(t):=e^{-i(t-t_n^j)\De}\lp f^j, \pq \lp N_n^j(t):=e^{i(t-t_n^j)\al D}\lp g^j,
 \pr \pe u_n := u_n - \sum_{j=1}^J \lp u_n^j, \pq \pe N_n := N_n - \sum_{j=1}^J \lp N_n^j.}
Then for any $j,k\in\{1 \ldots J\}$, we have
$t_\I^j:=\lim_{n\to\I} t_n^j \in \{0,\pm\I\}$,
\EQ{ \label{separation}
  j\not=k\implies \lim_{n\to\I}|t_n^j-t_n^k|=\I,}
\EQ{ \label{weak conv}
  \pt (\pe u_n,\pe N_n)(t_n^j)\to 0\ \text{weakly in $H^1\times L^2$} \text{as } n\to \I,
  \pr (\pe u_n,\pe N_n)(0)\to 0\ \text{weakly in $H^1\times L^2$} \text{as } n\to \I,}
and
\EQ{ \label{smallness}
 \lim_{J\to \bar{J}}\limsup_{n\to\I}
 [\|\pe u_n\|_{L^\I_t B^{-1/2-\de}_\I} + \|\pe N_n\|_{L^\I_t(\dot B^{-3/2-\de}_\I+\dot B^{-3/2+\de}_\I)}] =0.}
\end{lem}

\begin{rem}
1) \eqref{separation}--\eqref{weak conv} implies the linear orthogonality
\EQ{
 \pt \lim_{n\to\I}\|u_n(0)\|_{H^1}^2-\sum_{j=1}^J\|\lp u_n^j(0)\|_{H^1}^2-\|\pe u_n(0)\|_{H^1}^2 = 0,
 \pr \lim_{n\to\I} M(u_n(0))-\sum_{j=1}^J M(\lp u_n^j(0))-M(\pe u_n(0))=0,
 \pr \lim_{n\to\I}\|N_n(0)\|_{L^2}^2-\sum_{j=1}^J\|\lp N_n^j(0)\|_{L^2}^2-\|\pe N_n(0)\|_{L^2}^2 = 0,}
as well as the nonlinear orthogonality
\EQ{ \label{linear orth}
 \pt \lim_{n\to\I}\|u_n(0)\|_{L^4}^4-\sum_{j=1}^J\|\lp u_n^j(0)\|_{L^4}^4-\|\pe u_n(0)\|_{L^4}^4 = 0,
 \pr \lim_{n\to\I} E_S(u_n(0))-\sum_{j=1}^J E_S(\lp u_n^j(0))-E_S(\pe u_n(0))=0,
 \pr \lim_{n\to\I} K(u_n(0))-\sum_{j=1}^J K(\lp u_n^j(0))-K(\pe u_n(0))=0,
 \pr \lim_{n\to\I} E_Z(u_n(0),N_n(0))-\sum_{j=1}^J E_Z(\lp u_n^j(0),\lp N_n^j(0))-E_Z(\pe u_n(0),\pe N_n(0))=0.}
The same orthogonality holds also along $t=t_n^j$ instead of $t=0$.

2) The norms in \eqref{smallness} are related to the Sobolev embedding $L^2\subset \dot B^{-3/2}_\I$. Interpolation with the Strichartz estimate extends the smallness to any Strichartz norms as far as the exponents are not sharp either in $L^p$ or in regularity (including the low frequency of $N$).
\end{rem}

We call such a sequence of free solutions $\{(\lp u_n^j,\lp
N_n^j)\}_{n\in\N}$ a {\it free concentrating wave}. Now we introduce
the nonlinear profile associated to a free concentrating wave \EQ{
 (\lp{u}_n(t),\lp{N}_n(t))=U(t-t_n)(\lp{f},\lp{g}), \pq t_\I=\lim_{n\to\I}t_n\in\{0,\pm\I\},}
where $U(t)=e^{-it\De}\oplus e^{it\al D}$ denotes the free propagator. 
With it, we associate the {\it nonlinear profile} $(\np{u},\np{N})$,
defined as the solution of the Zakharov system satisfying \EQ{
 (u,N)=U(t)(\lp{f},\lp{g})+\int_{-t_\I}^t U(t-s)(nu,\al D|u|^2)(s)ds,}
which is obtained  by solving the initial data problem (if $t_\I=0$)
or by solving the final data problem (if $t_\I=\pm\I$). When
$t_\I=\pm\I$, the existence of wave operators  will be given at the
end of this paper as appendix .

We call $(\np{u}_n(t),\np{N}_n(t)):=(\np{u}(t-t_n),\np{N}(t-t_n))$
the {\it nonlinear concentrating wave} associated with
$(\lp{u}_n(t),\lp{N}_n(t))$. By the above construction we have \EQ{
 \pt\|(\lp{u}_n,\lp{N}_n)(0)-(\np{u}_n,\np{N}_n)(0)\|_{H^1\times L^2}
 \pr=\|(\np{u},\np{N})(-t_n)-U(-t_n)(\lp{f},\lp{g})\|_{H^1\times L^2}\to 0.}
Given a sequence of solutions to the Zakharov system with bounded
initial data, we can apply the free profile decomposition Lemma
\ref{free prof} to the sequence of initial data, and associate a
nonlinear profile with each free concentrating wave. If all
nonlinear profiles are scattering and the remainder is small enough,
then we can conclude that the original sequence of nonlinear
solutions is also scattering with a global Strichartz bound. More
precisely, we have
\begin{lem}\label{NL profile}
For each free concentrating wave $(\lp u_n^j,\lp N_n^j)$ in Lemma
\ref{free prof}, let $(\np u_n^j,\np N_n^j)$ be the associated
nonlinear concentrating wave. Let $(\tilde u_n, \tilde N_n)$ be the
sequence of nonlinear solutions with $(\tilde u_n, \tilde
N_n)(0)=(f_n,g_n)$. If $\|(\np
u_n^j,\np N_n^j)\|_{S(0,\I)}<\I$ for all $j< \bar J$, then
\EQ{\limsup_{n\to\I}\|(u_n,N_n)\|_{S(0,\I)}<\I.}
\end{lem}

To prove Lemma \ref{NL profile}, we need some global stability. In
the next subsection, we will refine the normal form reduction and
the nonlinear estimates that was used in \cite{GN}, and then prove
Lemma \ref{NL profile} and Lemma \ref{lem:crit}.

%%%%%%%%%%%%%%%%%%%%%%%%%%%%%%%%%%%%%%%%%%%%%%%%%%%%%%%%%%%%%%%%%%%%%%%%%%%%%%%%%%%%%%%%%%%%%%%%%%%%%%%%% stability

\subsection{Nonlinear estimates with small non-sharp norms}
In order to obtain the nonlinear profile decomposition, we need that
the non-sharp smallness \eqref{smallness} is sufficient to reduce
the nonlinear interactions globally. The idea is to use
interpolation, thus we need to do some refined estimates than in
\cite{GN}, more precisely, to avoid using the sharp (or endpoint) norms with $L^2_t$ or $L^\I_t$. 

\subsubsection{Modifying the nonresonant part} 
The first problem in following the Strichartz analysis in \cite{GN} is the $L^2_t$-type norms. In fact, one can observe that the use of $L^2_t$-type Strichartz norm for $N$
is inevitable for the low-high interactions of $nu$ in very low frequencies, since the regularity exponent
becomes bigger than that for the dual Schr\"odinger admissible exponent as we move the Strichartz norm of $N$ to $L^{2+}_t$.

However, this problem can be avoided by applying the normal form to those interactions.
In fact, there is no resonance in very low frequencies because
\EQ{
 -|\x|^2\pm \al|\x-\y|+|\y|^2 \sim \al|\x-\y|}
when all of $|\x|,|\x-\y|,|\y|$ are small.
Hence we include them into the ``non-resonant" interactions, which are integrated in time before the Strichartz estimate. 

The second problem is that our solution is no longer small, so the nonlinear terms without time integration (i.e. the boundary terms from the partial integration) do not contain any small factor for the perturbation argument. 
To overcome this difficulty, we shrink the ``non-resonant" part to either higher or lower frequencies, for which we gain a small factor, depending on the frequencies, from the regularity room. 
Hence our decomposition into the ``resonant" and ``non-resonant" interactions depends on the solution size. 

Thus we are lead to divide the bilinear interactions $nu$ and $|u|^2$ as follows. Let $u=\sum_{k\in\Z} P_ku$ be the standard homogeneous Littlewood-Paley decomposition such that $\supp\F P_k u\subset\{2^{k-1}<|\x|<2^{k+1}\}$. For a parameter $\be\ge 5+|\log_2\al|$, let 
\EQ{
 \pt XL:=\{(j,k)\in\Z^2 \mid j\ge \max(k+5,\be) \},
 \pr RL:=\{(j,k)\in\Z^2 \mid |j|<\be \tand j \ge k+5\}, 
 \pr LL:=\{(j,k)\in\Z^2 \mid \max(j,k)\le -\be\}, 
 \pr LH:=\{(j,k)\in\Z^2 \mid k>\min(j-5,-\be)\},
 \pr HH:=\{(j,k)\in\Z^2 \mid |j-k|<5 \tand \max(j,k)\ge\be\},
 \pr RR:=\{(j,k)\in\Z^2 \mid \max(j,k)<\be\},}
and $LX:=\{(k,j)\mid (j,k)\in XL\}$. Then 
\EQ{
 \Z^2=(XL \cup LL) \cup (RL \cup LH) = (XL\cup LX)\cup(HH\cup RR),}
where all the unions are disjoint. For any set $A\subset\Z^2$, and any functions $f(x),g(x)$, we denote the bilinear frequency cut-off to $A$ by 
\EQ{
 (fg)_A=\F^{-1}\int\cP_A\hat f(\x-\y)\hat g(\y)d\y:=\sum_{(j,k)\in A}(P_jf)(P_kg).}

For the nonlinear term $nu$, we apply the time integration by parts on $XL\cup LL$, where the phase factor $\om=-|\x|^2\pm\al|\x-\y|+|\y|^2$ is estimated
\EQ{
 |\om| \sim |\x-\y|\LR{\x-\y} \sim |\x-\y|\LR{\x},}
which is gained in the bilinear operator
\EQ{
 \pt\Om_\pm(f,g):=\F^{-1}\int \cP_{XL\cup LL}\frac{\hat f(\x-\y)\hat g(\y)}{-|\x|^2\pm\al|\x-\y|+|\y|^2}d\y, 
 \pr \Om(f,g):=\frac 12\{\Om_+(f,g)+\Om_-(\ba f,g)\}}
For the nonlinear term $u\bar{u}$, we integrate by parts on $XL\cup LX$. Then we get a bilinear operator of the form 
\begin{align}
\tilde\Om(f,g)&:=\F^{-1}\int \cP_{XL\cup LX}\frac{\hat f(\x-\y)\hat{\bar
g}(\y)}{|\xi-\eta|^2-|\eta|^2-\alpha|\xi|}d\y.
\end{align}

After this modification of the normal form, we can rewrite the
integral equation for \eqref{eq:Zak2} as follows. Let \EQ{
 \vec u:=(u,N), 
 \pq \vec u^0:=U(t)\vec u(0)=(e^{-it\De}u(0),e^{it\al D}N(0)).}
For the fixed free solution $\vec u^0$, the iteration $\vec
u'\mapsto \vec u$ is given by \EQ{
 \vec u=\vec u^0-U(t)B(\vec u(0),\vec u(0))+B(\vec u',\vec u')+Q(\vec u',\vec u')+T(\vec u',\vec u',\vec u'),}
where the bilinear forms $B,Q$ and the trilinear form $T$ are
defined by
\begin{align*}
B(\vec u_1,\vec u_2):=&(\Om(N_1,u_2),D\ti\Om(u_1,\ba u_2)),\\
Q(\vec u_1,\vec u_2):=&\int_0^t
U(t-s)((n_1u_2)_{LH\cup \al L},D(u_1\ba u_2)_{HH\cup\al L\cup L\al})(s)ds,\\
T(\vec u_1,\vec u_2,\vec
u_3):=&\int_0^tU(t-s)(\Om(D(u_1\ba u_2),u_3)+\Om(N_1,n_2u_3),D\ti
\Om(u_1,n_2u_3))(s)ds.
\end{align*}
For brevity, we denote
\begin{align*}
NL(\vec u_1,\vec u_2,\vec u_3):=&B(\vec u_1,\vec u_2)+Q(\vec u_1,\vec
u_2)+T(\vec u_1,\vec u_2,\vec u_3), \ \  NL(\vec u):=NL(\vec u,\vec
u), \\
B(\vec u):=&B(\vec u,\vec u),\ \ Q(\vec u):=Q(\vec u,\vec u),\ \
T(\vec u):=T(\vec u,\vec u,\vec u).
\end{align*}

We can estimate each term in the Duhamel formula using some powers
of Strichartz norms with non-sharp exponents. For brevity of
H\"older-type estimates, we denote the space-time norms by \EQ{
 \pt (b,d,s):=L^{1/b}_t \dot B^s_{1/d,2},
 \pr (b,d\pm \e,s)_+:=(b,d+\e,s)+(b,d-\e,s),
 \pr (b,d\pm \e,s)_\cap:=(b,d+\e,s)\cap(b,d-\e,s).}
Using the above notation, we introduce nearly full sets of the radial Strichartz norms for the Schr\"odinger and the wave equations. Fix small numbers
\EQ{
 0<\ka \ll \e \ll 1,}
and let
\EQ{ \label{def S}
 \pt SS:=\LR{D}^{-1}[(0,\fr 12,0)\cap(\fr 12,\fr 3{10}-\fr\ka3,\fr 25-\ka)],
 \pr SW:=(0,\fr 12,0)\cap(\fr 12,\fr 14-\fr\ka3,-\fr 14-\ka),
 \pq S:=SS\times SW.}
Also we denote the smallness in \eqref{smallness} by using
\EQ{
 \|u\|_X:=\|u\|_{L^\I_t(B^{-\fr 12-\de}_\I)}, \pq \|n\|_Y:=\|n\|_{L^\I_t(\dot{B}^{-\fr 32-\de}_\I+\dot{B}^{-\fr 32+\de}_\I)}, \pq Z:=X\times Y.}

In the nonlinear terms, we should choose appropriate Strichartz exponents so that all can be controled by interpolation between $S$ and $Z$.
For that purpose, we will choose $(b,d,s)$ for $u$ and $N$ respectively to be $H^s$ admissible with $0<s<1$ and $L^2$ admissible for radial functions.
Moreover, $b<1/2$ and $(b,d)\not=(0,1/2)$.
Besides that, we will use the sum space\footnote{This is because $N(0)\in L^2$ while $u(0)\in H^1=L^2\cap\dot H^1$.} with small $\e>0$ for $N$ and the intersection for $u$, so that we can dispose of very low or high frequencies, and sum over the dyadic decomposition without any difficulty.

\subsubsection{Bare bilinear terms}
First consider the bilinear terms which do not contain the time integration, namely the boundary term in the transform.
In the equation for $u$, $\Om(n,u)$ is roughly like $\LR{D}^{-1}(D^{-1}n)u$ for each dyadic piece.
\begin{lem}
(a) There exists $\theta>0$ such that for any $N$ and $u$, we have
\begin{align}
 \|\Om(n,u)\|_{L^\I H^1} \lec& 2^{-\te \be}\|u\|_{SS}^{1-\te}\|n\|_{SW}^{1-\te}\|n\|_Y^\te
 \|u\|_X^\te,\\
 \|\Om(n,u)\|_{SS} \lec& 2^{-\te \be}\|u\|_{SS}^{1-\te}\|n\|_{SW}^{1-\te}\|n\|_Y^\te \|u\|_X^\te.
\end{align}

(b) There exists $\theta>0$ such that for any $u$ and $u'$, we have
\begin{align}
 \|D\ti\Om(u,u')\|_{L^\I L^2} \lec& 2^{-\te \be}\|u\|_{SS}^{1-\te}\|u'\|_{SS}^{1-\te}\|u\|_X^{\te}\|u'\|_X^\te,\\
 \|D\ti\Om(u,u')\|_{SW} \lec& 2^{-\te \be}\|u\|_{SS}^{1-\te}\|u'\|_{SS}^{1-\te}\|u\|_X^{\te}\|u'\|_X^\te.
\end{align}
\end{lem}

\begin{proof}
(a) By the Coifman-Meyer-type bilinear estimate on dyadic pieces (see \cite[Lemma 3.5]{GN}), we have for $(j,k)\in XL$, 
\begin{align*}
\|\Om(n_j,u_k)\|_{L^\I H^1} &\lesssim \|D^{-1}n_j\|_{(0,\fr 15\pm\e,0)_+}\|u_k\|_{(0,\fr 3{10}\pm\e,0)_\cap}\\
 &\lesssim 2^{-\be/10}\|D^{-1}n_j\|_{(0,\fr 12\pm\e,1)_+}\|u_k\|_{(0,\fr 3{10}\pm\e,0)_\cap},
\end{align*}
and for $(j,k)\in LL$, 
\begin{align*}
\|\Om(n_j,u_k)\|_{L^\I H^1} &\lesssim \|D^{-1}n_j\|_{(0,\fr 2{15}\pm\e,0)_+}\|u_k\|_{(0,\fr {11}{30}\pm\e,0)_\cap}\\
 &\lesssim 2^{-\be/10}\|D^{-1}n_j\|_{(0,\fr 12\pm\e,1)_+}\|u_k\|_{(0,\fr {11}{30}\pm\e,0)_\cap}. 
\end{align*}
Since the right hand side is bounded by $\|n\|_{L^\I L^2}\|u\|_{L^\I H^1}$ via non-sharp Sobolev embedding, we obtain, after summation over dyadic decomposition,
\EQ{
 \|\Om(n,u)\|_{L^\I H^1} \lec 2^{-\be/10}\|u\|_{SS}^{1-\te}\|n\|_{SW}^{1-\te}\|n\|_Y^\te \|u\|_X^\te,}
for some small $\te>0$. Similarly we have, for $(j,k)\in XL$, 
\EQ{
 \pt\|\Om(n_j,u_k)\|_{\LR{D}^{-1}(\fr 12,\fr 3{10}-\fr\ka3,\fr 25-\ka)}
 \pn\lec \|D^{-1}n_j\|_{(\fr 14,\fr {7}{30}-\fr \ka 3\pm\e,\fr 25-\ka)_+}\|u_k\|_{(\fr 14,\fr 1{15}\pm\e,0)_\cap}
 \pr\pq\lec 2^{-\be/20}\|D^{-1}n_j\|_{(\fr 14,\fr {7}{30}-\fr \ka 3\pm\e,\fr{9}{20}-\ka)_+}\|u_k\|_{(\fr 14,\fr 1{15}\pm\e,0)_\cap},}
and for $(j,k)\in LL$,
\EQ{
 \pt\|\Om(n_j,u_k)\|_{\LR{D}^{-1}(\fr 12,\fr 3{10}-\fr\ka3,0)}
 \pn\lec \|D^{-1}n_j\|_{(\fr 14,\fr 1{15}-\fr \ka 3\pm\e,0)_+}\|u_k\|_{(\fr 14,\fr {7}{30}\pm\e,0)_\cap}
 \pr\pq\lec 2^{-\be/20}\|D^{-1}n_j\|_{(\fr 14,\fr 1{15}-\fr \ka 3\pm\e,-\frac{1}{20})_+}\|u_k\|_{(\fr 14,\fr {7}{30}\pm\e,0)_\cap} }
Hence in either case we can control by non-sharp norms, so
\EQ{
 \|\Om(n,u)\|_{SS} \lec 2^{-\be/20}\|u\|_{SS}^{1-\te}\|n\|_{SW}^{1-\te}\|n\|_Y^\te \|u\|_X^\te.}

(b) We may assume $(j,k)\in XL$, since the other case $LX$ is treated in the same way. Similarly to the above, we have 
$D\ti\Om(f_j,g_k)\sim\LR{D}^{-1}(f_jg_k)$, so 
\EQ{
 \pn\|D\ti\Om(u_j,u'_k)\|_{L^\I L^2} \pt\lec \|\LR{D}D\ti\Om(u_j,u'_k)\|_{L^\I(L^2+L^{6/5})}
 \pr\lec \|u_j\|_{(0,\fr 13\pm\e,0)_+}\|u'_k\|_{(0,\fr 13\pm\e,0)_+}
 \pr \lec 2^{-\be/10} \|u_j\|_{(0,\fr 13\pm\e,\frac{1}{10})_+}\|u'_k\|_{(0,\fr 13\pm\e,0)_+} ,}
hence
\EQ{
 \|D\ti\Om(u,u')\|_{L^\I L^2} \lec 2^{-\be/10}\|u\|_{SS}^{1-\te}\|u'\|_{SS}^{1-\te}\|u\|_X^{\te}\|u'\|_X^\te.}
Similarly,
\EQ{
 \pn\|D\ti\Om(u_j,u'_k)\|_{(\fr 12,\fr 14-\fr\ka3,-\fr 14-\ka)}
 \pt\lec\|\LR{D}D\ti\Om(u_j,u'_k)\|_{(\fr 12,\fr 23,0)}
 \pr\lec 2^{-\be/10}\|u_j\|_{(\fr 14,\fr 13,\frac1{10})}\|u'_k\|_{(\fr 14,\fr 13,0)},}
and so \EQ{ \|D\ti\Om(u,u')\|_{SW} \lec
 2^{-\be/10}\|u\|_{SS}^{1-\te}\|u'\|_{SS}^{1-\te}\|u\|_X^{\te}\|u'\|_X^\te.}
Thus the proof is completed.
\end{proof}

\subsubsection{Duhamel bilinear terms}
Next we consider the remaining bilinear terms in the Duhamel form after the normal form transform. 
Here we have to use the radial improvement of the Strichartz norms. 
For brevity, we denote the integrals in the Duhamel formula by 
\EQ{
 I_u f:=\int_0^t e^{-i(t-s)\De}f(s)ds, \pq I_Nf:=\int_0^t e^{i(t-s)\al D}f(s)ds.}

\begin{lem}
(a) There exists $\theta>0$ and $C(\be)>1$ such that for any $N$ and $u$, we have
\begin{align*}
\|I_u(nu)_{LH}\|_{SS}
 \le &C(\be) \|u\|_{SS}^{1-\te}\|n\|_{SW}^{1-\te}\|n\|_Y^{\te}
 \|u\|_X^{\te},\\
\|I_u(nu)_{RL}\|_{SS}
 \le& C(\be) \|u\|_{SS}^{1-\te}\|n\|_{SW}^{1-\te}\|n\|_Y^{\te} \|u\|_X^{\te}.
\end{align*}

(b) There exists $\theta>0$ and $C(\be)>1$ such that for any $u$ and $u'$, we have
\begin{align*}
\|I_ND(uu')_{HH}\|_{SW} \le& C(\be)
\|u\|_{SS}^{1-\te}\|u'\|_{SS}^{1-\te}\|u\|_X^{\te}\|u'\|_X^{\te},\\
\|I_ND(uu')_{RR}\|_{SW} \le& C(\be)
\|u\|_{SS}^{1-\te}\|u'\|_{SS}^{1-\te}\|u\|_X^{\te}\|u'\|_X^{\te}.
\end{align*}
\end{lem}

\begin{proof}
In this proof we ignore the dependence of the constants on $\be$. 

(a) For $(j,k)\in LH$, we have for $0\le s\le 1$, 
\EQ{
 \|n_ju_k\|_{(1-2\e,\fr 12+2\e,s+2\e)}
 \pt\lec \|n_j\|_{(\fr 12-\e,\fr 14\pm\fr {\e} 3,-\fr 14-\e)_\pm}
  \|u_k\|_{(\fr 12-\e,\fr 14+2\e\pm\fr {\e} 3,s+\fr 14+3\e)_\cap}
 \pr\lec \|n_j\|_{(\fr 12-\e,\fr 14\pm\fr {\e} 3,-\fr 14-\e)_\pm}
  \|u_k\|_{(\fr 12-\e,\fr 14+2\e\pm\fr {\e} 3,\fr 54+3\e)_\cap},}
where in the second inequality we used that $k$ is bounded from below. 
Since the left hand side is $\dot H^s$-admissible norm for the Strichartz estimate (without the radial symmetry), we obtain the full Strichartz bound in $H^1$. 

For $(j,k)\in RL$, we may neglect the regularity of $n_j$ and the product, since their frequencies are bounded from above and below. Using the radial improved Strichartz, the full $H^1$ Strichartz norm is bounded by 
\EQ{
 \|n_ju_k\|_{(\fr 12+2\e,\fr 34-3\e,0)}
 \lec \|n_j\|_{(\fr 12-\e,\fr 14,0)}
  \|u_k\|_{(3\e,\fr 12-3\e,0)}.}

Summing these estimates over dyadic pieces in the specified regions, and using non-sharp Sobolev embedding and interpolation, we obtain
\EQ{
 \pt\|I_u(nu)_{LH}\|_{SS}
 \lec \|u\|_{SS}^{1-\te}\|n\|_{SW}^{1-\te}\|n\|_Y^{\te} \|u\|_X^{\te},
 \pr\|I_u(nu)_{R L}\|_{SS}
 \lec \|u\|_{SS}^{1-\te}\|n\|_{SW}^{1-\te}\|n\|_Y^{\te} \|u\|_X^{\te}.}

(b) We consider only the case $j\ge k$ for $u_j u_k'$, since the other case is treated in the same way. For $(j,k)\in HH$, 
\EQ{
 \|u_j u'_k\|_{(1-\e,\fr 12+\fr 23\e,1+\e)}
 \lec \|u_j\|_{(\fr 12-\fr \e 2,\fr 14+\fr \e 3,\fr 12+\fr \e2)} \|u'_k\|_{(\fr 12-\fr \e 2,\fr 14+\fr \e 3,\fr 12+\fr \e2)},}
and in the case $(j,k)\in RR$, since $j$ is bounded from above, 
\EQ{
  \|u_j u'_k\|_{(\fr 12+\e,\fr 34,\fr 34+\e)}
 \pt\lec \|u_j\|_{(\fr 12-\e,\fr 14+2\e,\fr 34+\e)} \|u'_k\|_{(2\e,\fr 12-2\e,0)}
 \pr\lec \|u_j\|_{(\fr 12-\e,\fr 14+2\e,\fr 12)} \|u'_k\|_{(2\e,\fr 12-2\e,0)}.}
Hence
\EQ{
 \pt \|D(uu')_{HH}\|_{(1-\e,\fr 12+\fr 23\e,\e)} \lec \|u\|_{SS}^{1-\te}\|u'\|_{SS}^{1-\te}\|u\|_X^{\te}\|u'\|_X^{\te},
 \pr \|D(uu')_{\al L+L\al}\|_{(\fr 12+\e,\fr 34,-\fr 14+\e)} \lec \|u\|_{SS}^{1-\te}\|u'\|_{SS}^{1-\te}\|u\|_X^{\te}\|u'\|_X^{\te}.}
The left hand sides are $L^2$-admissible norms for radial functions. Thus the proof is completed by the radial improved Strichartz. 
\end{proof}

\subsubsection{Duhamel trilinear terms}
Finally we estimate the trilinear terms which appear after the
normal transform. These are supposedly the easiest, but there is a
small complication due to the fact that we have to use negative
Sobolev spaces for $N$ in some of the products: \EQ{
 \pt \|fg\|_{\dot B^{-s}_{r,2}} \lec \|f\|_{\dot B^{-s}_{p,2}}\|g\|_{\dot B^s_{q,2}}
 \pr 0\le s<3/q,\ 1/r=1/p+1/q-s/3.}
In the next lemma, the constant may decay as $\be\to\I$, but we do not need it. 
\begin{lem}
(a) There exists $\theta>0$ such that for any $u,v,w,n,n'$, we have
\begin{align*}
 \|I_u\Om(D(uv),w)\|_{SS} \lec&
 \|u\|_{SS}^{1-\te}\|v\|_{SS}^{1-\te}\|w\|_{SS}^{1-\te}\|u\|_X^{\te}\|v\|_X^{\te}\|w\|_X^{\te}.\\
 \|I_u\Om(n,n'u)\|_{SS}
 \lec& \|n\|_{SW}^{1-\te}\|n'\|_{SW}^{1-\te}\|u\|_{SS}^{1-\te}\|n\|_Y^{\te}\|n'\|_Y^{\te}\|u\|_X^{\te}.
\end{align*}

(b) There exists $\theta>0$ such that for any $n,u,u'$, we have
\begin{align*}
\|I_ND\ti\Om(nu,u')\|_{SW}+\|I_ND\ti\Om(u,nu')\|_{SW} \lec
\|n\|_{SW}^{1-\te}\|u\|_{SS}^{1-\te}\|u'\|_{SS}^{1-\te}\|n\|_Y^{\te}\|u\|_X^{\te}\|u'\|_X^{\te}.
\end{align*}
\end{lem}

\begin{proof}
(a) Since $\Om(D(uv)_j,w_k)\sim \LR{D}^{-1}((uv)_jw_k)$,
\EQ{
 \|\Om(D(uv)_j,w_k)\|_{L^1H^1} \lec \|u\|_{L^3L^6}\|v\|_{L^3L^6}\|w\|_{L^3L^6},}
and by  non-sharp Sobolev embedding and interpolation,
\EQ{
 \|\Om(D(uv),w)\|_{L^1H^1} \lec \|u\|_{SS}^{1-\te}\|v\|_{SS}^{1-\te}\|w\|_{SS}^{1-\te}\|u\|_X^{\te}\|v\|_X^{\te}\|w\|_X^{\te}.}
For $\Om(n_j,(n'u)_k)$, we have either $2^j\gg 2^k$ or $2^j+2^k\ll 1$. In the first case, we have
\EQ{
 \pt\|\Om(n_j,(n'u)_k)\|_{L^1H^1}
 \lec \|D^{-1+5\e}n_j\|_{(\fr 12-\e,2\e\pm\fr {\e}6,0)_\pm} \|D^{-5\e}(n'u)_k\|_{(\fr 12+\e,\fr 12-2\e\pm\fr {\e}6,0)_\cap}
 \pr\lec \|n_j\|_{(\fr 12-\e,2\e\pm\fr {\e}6,-1+5\e)_\pm} \|n'\|_{(2\e,\fr 12-\fr 73\e\pm\fr {\e}6,-5\e)_\pm} \|u\|_{(\fr 12-\e,2\e\pm \fr {\e}3,5\e)_\cap},}
where we used the product estimate for negative Sobolev spaces for $n'u$.
In the second case $2^j+2^k\ll 1$, we have
\EQ{
 \pt\|\Om(n_j,(n'u)_k)\|_{L^1H^1}
 \lec \|D^{-1}n_j\|_{(\fr 12-\e,\fr \e3\pm\fr {\e}6,0)_\pm} \|(n'u)_k\|_{(\fr 12+\e,\fr 12-\fr \e3\pm\fr {\e}6,0)_\cap}
 \pr\lec \|n_j\|_{(\fr 12-\e,\fr \e3\pm\fr {\e}6,-1)_\pm} \|n'\|_{(2\e,\fr 12-\fr 73\e\pm\fr {\e}6,-5\e)_\pm} \|u\|_{(\fr 12-\e,\fr {11}3\e\pm \fr {\e}3,5\e)_\cap}.}
Hence, by  non-sharp Sobolev embedding and interpolation,
\EQ{
 \|\Om(n,n'u)\|_{L^1H^1}
 \lec \|n\|_{SW}^{1-\te}\|n'\|_{SW}^{1-\te}\|u\|_{SS}^{1-\te}\|n\|_Y^{\te}\|n'\|_Y^{\te}\|u\|_X^{\te}.}

(b) We have $D\ti \Om\sim\LR{D}^{-1}$ on each dyadic piece, so
\EQ{
 \pt\|D\ti\Om((nu)_j,u'_k)\|_{L^1L^2} \lec \|\ti\Om((nu)_j,u'_k)\|_{(1,\fr 56-\fr 53\e,-5\e)}
 \pr\lec \|(nu)_j\|_{(\fr 12+\e,\fr 23-\e,-5\e)}\|u'_k\|_{(\fr 12-\e,\fr 16+\e,5\e)}
 \pr\lec \|n\|_{(2\e,\fr 12-\fr 73\e\pm\fr {\e}6,-5\e)_\pm}
 \|u\|_{(\fr 12-\e,\fr 16+\e\pm\fr {\e}6,5\e)_\cap}\|u'_k\|_{(\fr 12-\e,\fr 16+\e,5\e)},}
where we used the product estimate twice, but did not use any restriction on $j,k$. Hence we have the same estimate on $\ti\Om(u_j,(n'u)_k)$, and so
\EQ{
 \pt\|D\ti\Om(nu,u')\|_{L^1L^2}+\|D\ti\Om(u,nu')\|_{L^1L^2}
 \pr\lec \|n\|_{SW}^{1-\te}\|u\|_{SS}^{1-\te}\|u'\|_{SS}^{1-\te}\|n\|_Y^{\te}\|u\|_X^{\te}\|u'\|_X^{\te}.}
Thus, the proof is completed.
\end{proof}

Note that in the above estimates we needed the $L^\I_t$-type norms only for the bare bilinear terms, but not for the Duhamel terms. Thus we have obtained 
\begin{lem}\label{B-Q-T}
There exist $\te>0$, $\y>0$ and $C(\be)>1$ such that for each $\be\gg 1$ and any $\vec u_1,\vec u_2,\vec u_3$, we have
\EQ{\label{NL est}
 \pt2^{\te\be}\|B(\vec u_1,\vec u_2)\|_S+\|Q(\vec u_1,\vec u_2)\|_S/C(\be) \lec \|\vec u_1\|_S^{1-\te}\|\vec u_2\|_S^{1-\te}\|\vec u_1\|_Z^\te\|\vec u_2\|_Z^\te,
 \pr\|T(\vec u_1,\vec u_2,\vec u_3)\|_S \lec \|\vec u_1\|_S^{1-\te}\|\vec u_2\|_S^{1-\te}\|\vec u_3\|_S^{1-\te}\|\vec u_1\|_Z^\te\|\vec u_2\|_Z^\te\|\vec u_3\|_Z^\te.}
For the Duhamel terms we have also 
\EQ{\label{Q-T}
 \pt\|Q(\vec u_1,\vec u_2)\|_S \lec C(\be)\|\vec u_1\|_{\widetilde{S}}\|\vec u_2\|_{\widetilde{S}},
 \pr\|T(\vec u_1,\vec u_2,\vec u_3)\|_S \lec \|\vec u_1\|_{\widetilde{S}}\|\vec u_2\|_{\widetilde{S}}\|\vec u_3\|_{\widetilde{S}},}
 where
 \EQ{
 \pt \widetilde{S}:=\widetilde{SS}\times \widetilde{SW},
 \pr \widetilde{SS}:=\LR{D}^{-1}[(\eta,\fr 12-\fr 25\eta,
 \fr 45\eta)\cap(\fr 12,\fr 3{10}-\fr\ka3,\fr 25-\ka)],
 \pr \widetilde{SW}:=(\eta,\fr 12-\fr 12\eta,
 -\fr 14\eta)\cap(\fr 12,\fr 14-\fr\ka3,-\fr 14-\ka).
 }
\end{lem}

\subsection{Nonlinear profile approximation} We will prove Lemma \ref{NL
profile} by the following two lemmas.

\begin{lem}[Stability]\label{Stability}
For any $A>0$ and $\sigma>0$, there exists $\varsigma>0$ with the following property: 
Suppose that $ \vec  u_a$ satisfies
$\|\vec u_a\|_{S(0,\I)}\leq A$ and approximately solves the Zakharov system in the sense that 
\begin{equation*}
\begin{split}
\vec u_a=U(t)\vec u_a(0)-U(t)B(\vec u_a(0))+NL(\vec u_a)+\vec e\\
\end{split}
\end{equation*}
and $\|\vec e\|_{S(0,\I)}\leq \varsigma$. 
Then for any initail data $\vec u(0)$ satisfying $\|\vec u(0)-\vec u_a(0)\|_{H^1\times L^2}<\varsigma$, there is a unique global solution $\vec u$ satisfying 
$\|\vec u-\vec u_a\|_{S(0,\I)}< \sigma$.
\end{lem}
%%%%%%%%%%%%%%%%%%%%%%%%%%%%%%%%%%%%%%%%%%%%%%%%%%%%%%%%%%%%%%%%%%%%%%%%%%%%%%%%%%%%%%%%%%%%%%%% the proof of stability
\begin{proof}
Denote $\vec u_{\rhd}=\vec u_a-\vec u$, 
then $\|\vec u_{\rhd}(0)\|_{H^1\times L^2}\leq \varsigma$ and
\EQ{
\begin{split}
\vec u_{\rhd}=&U(t)\vec u_{\rhd}(0)-U(t)B(\vec u_a(0))+NL(\vec u_a)+\vec e
+U(t)B(\vec u(0))-NL(\vec u).
 \end{split}
}
Thus
\EQ{
\begin{split}
\|\vec u_{\rhd}\|_{S}\lesssim 2 \varsigma+\|B(\vec u_a)-B(\vec u)\|_{S}
+\|Q(\vec u_a)-Q(\vec u)\|_{S}
+\|T(\vec u_a)-T(\vec u)\|_{S}.
\end{split}
}

Noting that $Z\supset S$, by \eqref{NL est} we have
\begin{align*}
\|B(\vec u_a)-B(\vec u)\|_{S}
&\leq\|B(\vec u_a,\vec u_\rhd)\|_{S}+\|B(\vec u_\rhd,\vec u_a)\|_{S}+\|B(\vec u_\rhd,\vec u_\rhd)\|_{S}\\
&\lesssim 2^{-\te\be} A\|\vec u_{\rhd}\|_{S}+2^{-\te\be}\|\vec u_{\rhd}\|^{2}_{S}.
\end{align*}
By \eqref{Q-T}, we have
\EQ{
\begin{split}
\|Q(\vec u_a)-Q(\vec u)\|_{S}\leq& C(\beta) \left(\|\vec u_a\|_{\widetilde{S}}\|\vec u_{\rhd}\|_{S}+\|\vec u_{\rhd}\|^{2}_{S}\right),\\
\|T(\vec u_a)-T(\vec u)\|_{S}\leq& C\left(\|\vec u_a\|^2_{\widetilde{S}}\|\vec u_{\rhd}\|_{S}+ \|\vec u_a\|_{\widetilde{S}}\|\vec u_{\rhd}\|^{2}_{S} +\|\vec u_{\rhd}\|^{3}_{S}\right).
\end{split}
}
So
\EQ{\label{diff1}
\begin{split}
\|\vec u_{\rhd}\|_{S}\leq 2\varsigma&+(2^{-\te\be}C A+C(\beta)\|\vec u_a\|_{\widetilde{S}}+C\|\vec u_a\|_{\widetilde{S}}^{2})\|\vec u_{\rhd}\|_{S}\\
&
+ (2^{-\te\be}C+C(\beta)+C\|\vec u_a\|_{\widetilde{S}})\|\vec u_{\rhd}\|^{2}_{S}+C\|\vec u_{\rhd}\|^{3}_{S}.
\end{split}
}
Choose $\beta=\beta(A)$ such that $2^{-\te\beta}CA<\fr14$.
Then we  subdivide the time interval $[0,\I)$ into finite subintervals $I_j=[t_j, t_{j+1}]$, $j=1,\cdots, J$, $J=J(A,\beta)$ such that
\EQ{C(\beta)\|\vec u_a\|_{\widetilde{S}(I_j)}+C\|\vec u_a\|^{2}_{\widetilde{S}(I_j)}<\fr14}
for each $j$.
 Let $\varsigma=\varsigma(A, \sigma, \beta,J)$ small such that
\EQ{C(\beta)8^{2J}\varsigma\ll1,\ \ 8^{2J}\varsigma \ll \sigma.}
Then by \eqref{diff1} on $I_1$, we have $\|\vec u_{\rhd}\|_{S(I_1)}\leq 8\varsigma$ and
\begin{align*}
 &\|\vec u_{\rhd}(t_2)\|_{H^1\times L^2}\\
 \leq &\|U(t_2-t_1)\vec u_{\rhd}(t_1)\|_{H^1\times L^2}
 +\|U(t_2-t_1)B(\vec u_a(t_1))-U(t_2-t_1)B(\vec u(t_1))\|_{H^1\times L^2}\\
 &+\|B(\vec u_a(t_2))-B(\vec u(t_2))\|_{H^1\times L^2}+\|Q(\vec u_a)-Q(\vec u)\|_{S(I_1)}\\
 &+\|T(\vec u_a)-T(\vec u)\|_{S(I_1)}+\|\vec e\|_{S(I_1)}\\
 \leq&
 2\varsigma+4\cdot8\varsigma  \leq8^2\varsigma.
\end{align*}
Using the same analysis as above, we can get  $\|\vec u_{\rhd}\|_{S(I_2)}\leq 8^3\varsigma$.
Iterating this for $I_2,I_3\etc I_J$, 
 we obtain $\|\vec{u}_a-\vec{u}_1\|_{S}\lec 8^{2J}\varsigma \ll \sigma$, the desired result was obtained.
\end{proof}

With $J$ close to $\bar{J}$ and large $n$, our approximate solution is given by 
\EQ{
 \vec{u}_n^J=(u_n^J,N_n^J):= \sum_{j=1}^J(\np u_n^j,\np N_n^j)+(\pe u_n,\pe N_n).}
To prove Lemma \ref{NL profile}, we only need to  prove that $\vec u_n^J$ is an approximate solution of the Zakharov system. In fact, we have
\begin{lem}\label{per}
Suppose that $\|(\np u_n^j,\np N_n^j)\|_S<\I$ for all $j<\bar J$, then
\begin{align*}
\lim_{J\to \bar{J}}\limsup_{n\to\I} \|U(t)B(\vec u_n^J(0))-NL(\vec u_n^J) -\sum_{j=1}^J[U(t)B(\vec{\np u}_n^j(0))-NL(\vec{\np u}_n^j)]\|_{S}=0.
\end{align*}
\end{lem}
Note that $\|(\np u_n^j,\np N_n^j)\|_S$ does not depend on $n$. 
\begin{proof}
By triangle inequality, it suffices to show that
\EQ{\label{cross term}
\begin{split}
\lim_{n\to\I}&\|\sum_{j\le J}[U(t)B(\vec{\np u}_n^j(0))-NL(\vec{\np u}_n^j)]\\
&-[U(t)B(\sum_{j\le J}\vec{\np u}_n^j(0))
   -NL(\sum_{j\le J}\vec{\np u}_n^j)]\|_{S}=0,
\end{split}}
and
\EQ{\label{diff2}
\begin{split}
 \lim_{J\to \bar{J}}\limsup_{n\to\I}&\|
 [U(t)B(\vec u^J_n(0))-NL(\vec u^J_n)]\\
 &-[U(t)B(\vec u^J_n(0)-\pe{\vec u}_n(0))-NL(\vec u^J_n-\pe{\vec u}_n)]\|_{S}= 0.
\end{split}
}

In fact,
\begin{align*}
L.H.S \ of\  \eqref{cross term}
\lesssim  \sum_{i\neq j}\left(\|B(\vec{\np u}_n^i, \vec{\np u}_n^j)\|_{S}
+\|Q(\vec{\np u}_n^i, \vec{\np u}_n^j)\|_{S}\right)+\sum_{i\neq j\text{ or }j\neq k}\|T(\vec{\np u}_n^i, \vec{\np u}_n^j,\vec{\np u}_n^k)\|_{S}.
\end{align*}
For each $i\not=j$, we have $|t_n^i-t_n^j|\to \I$. for the subsequence $t_n^i-t_n^j\to\I$, we have by \eqref{NL est}, 
\EQ{\label{B-decay}
&\|B(\vec{\np u}_n^i, \vec{\np u}_n^j)\|_{S}
\lec \|B(\vec{\np u}^i(\cdot-t_n^i), \vec{\np u}^j(\cdot-t_n^j))\|_{S(-\I,(t_n^i+t_n^j)/2)\cap S((t_n^i+t_n^j)/2,\I)}
 \pr\lec \|\vec{\np u}^i\|_S^{1-\te}\|\vec{\np u}^j\|_S^{1-\te}\left[\|\vec{\np u}^i\|_{Z(-\I,t_n^j-t_n^i)/2)}^\te \|\vec{\np u}^j\|_Z^\te +\|\vec{\np u}^i\|_Z^\te \|\vec{\np u}^j\|_{Z((t_n^i-t_n^j)/2,\I)}^\te \right].}
For each $j$, by the scattering of $\vec{\np u}^j$,
\EQ{\label{decay}
 \lim_{T\to\I}\|\vec{\np u}^j\|_{Z(|t|\geq T)}=0,}
so from the above estimate 
\EQ{
 \|B(\vec{\np u}_n^i, \vec{\np u}_n^j)\|_{S}\to 0,}
as $t_n^i-t_n^j\to\I$. The case $t_n^i-t_n^j\to-\I$ is treated similarly, as well as the other terms $Q$ and $T$. 
Thus we obtain 
\EQ{
\begin{split}
&\|B(\vec{\np u}_n^i, \vec{\np u}_n^j)\|_{S}\to 0 \text{ for } i\neq j, \\
&\|Q(\vec{\np u}_n^i, \vec{\np u}_n^j)\|_{S}\to 0 \text{ for } i\neq j, \\
&\|T(\vec{\np u}_n^i, \vec{\np u}_n^j,\vec{\np u}_n^k)\|_{S}\to 0  \text{ for } i\neq j  \text{ or }i=j\neq k, 
\end{split}
}
from which \eqref{cross term} follows immediately. 

In order to prove \eqref{diff2}, we need a uniform bound on the approximate solutions $\vec u_n^J$ for $J\to\bar J$. 
Note that \eqref{linear orth} implies that $\|\vec{\lp u_n}^j(0)\|_{H^1\times L^2}\ll 1$ except for a bounded number of $j$. Let $A$ be the set of $j$ in the latter case. Then for all $j\not\in A$, the small data scattering implies that 
\EQ{
 \|(\np u_n^j,\np N_n^j)\|_S\lesssim \|(\lp u_n^j(0),\lp N_n^j(0))\|_{H_x^1\times L_x^2} \ll 1. } 
Then by the orthogonality in $H_x^1\times L_x^2$ and $|t_n^i-t_n^j|\to\I$, we deduce 
\EQ{
 \|\sum_{j\not\in A}\vec{\np u}_n^j\|^2_S&\lesssim \sum_{j\not\in A}\|\vec{\np u}_n^j\|^2_S
 \lesssim \sum_{j\not\in A}\|\vec{\lp u}_n^j\|^2_{H_x^1\times L_x^2}\lec 1.}
Since the number of the remaining components $j\in A$ are bounded, we obtain 
\EQ{ \label{unif bound on approx}
 \sup_{J<\ba J} \sup_n \|\vec u_n^J\|_S <\I.}
The left hand side of \eqref{diff2} is bounded by 
\EQ{
 &\|B(\vec u^J_n)-
B(\vec u^J_n-\pe{\vec u}_n)
\|_{S}
+\|Q(\vec u^J_n)
-Q(\vec u^J_n-\pe{\vec u}_n)\|_{S}\\
&+\| T(\vec u^J_n)-T(\vec u^J_n-\pe{\vec u}_n)\|_{S} .}
By \eqref{NL est}, \eqref{smallness} and \eqref{unif bound on approx},
\EQ{
 \lim_{J\to \bar{J}}\limsup_{n\to\I}\|B(\vec u^J_n)-B(\vec u^J_n-\pe{\vec u}_n)\|_{S}= 0.
}
One can estimates  $Q$ and $T$ similarly. Then  \eqref{diff2} was proved.
\end{proof}

\begin{proof}[Proof of Lemma \ref{NL profile}]
By the construction of $\vec u^J_n$,
\EQ{
 \lim_{n\to\infty}\|\vec u^J_n(0)-\vec{u}_n(0)\|_{H^1\times L^2}= 0.}
By Lemma \ref{per} and Lemma \ref{Stability}, passing to a subsequence if necessary, we obtain $\|\vec u^J_n-\vec{u}_n\|_{S}\ll 1$ for large $J$ and $n$. 
\end{proof}

\begin{proof}[Proof of Lemma \ref{lem:crit}]
By the definition of $E_\la^*$, there is a sequence of
global solutions $(u_n,N_n)$ in $\K^+_\la(a)$ such that
 \EQ{
 \lim_{n\rightarrow\infty}\mathscr{E}_\la(u_n,N_n)=E_\la^*, \  \lim_{n\rightarrow\infty}\|(u_n,N_n)\|_{S(-\I,0)}=\lim_{n\rightarrow\infty}\|(u_n,N_n)\|_{S(0,\I)}=\I.
 }
To see this, first note that since $(u_n,N_n)$ are in $\K^+_\la(a)$, they are bounded in $H^1\times L^2$ by the energy. Hence if $\|(u_n,N_n)\|_S\to\I$ then the $L^2_t$ part must diverge, and we can translate $(u_n,N_n)$  in $t$ so that the norm diverges both on $(-\I,0)$ and on $(0,\I)$. 

For the sequence $(u_n(0),N_n(0))$, we use the linear profile
decomposition. For the associated nonlinear profile $(\np u_n^j,\np
N_n^j)$, we must have $K(\np u_n^j(0))\geq0$ for each $j$. In
fact, if we denote \EQ{G_\la(\varphi)=J_\la(\varphi)-\fr 13 K=
\left(\frac{1}{2}-\frac{1}{3}\right)\|\na u\|_2^2+\fr
{\la^2}2\|u\|_2^2>0,}
 then
\EQ{ \label{char Q}
J_\la(Q_\la)&=\inf\{J_\la(\fy)\mid \fy\not=0,\ K(\fy)=0\}\\
&=\inf\{G_\la(\fy)\mid \fy\not=0,\ K(\fy)=0\}\\
&=\inf\{G_\la(\fy)\mid \fy\not=0,\ K(\fy)\leq0\}.} 
By the orthogonality, 
\EQ{
 \varlimsup_{n\to\I} G_\la(u_n(0))= \varlimsup_{n\to\I} \left(\sum_{j=1}^J
G_\la(\np u_n^j(0))+G_\la(\pe u_n(0))\right)\leq \la E_\la^*<
J_\la(Q_\la).} 
Hence, for $n$ sufficiently large, $G_\la(\np u_n^j(0))<
J_\la(Q_\la)$; and then by the third line of \eqref{char Q},
$K(\np u_n^j(0))\geq0$. Noting that 
\EQ{
\lim_{n\to\I} \mathscr{E}_\la(u_n(0),N_n(0))-\sum_{j=1}^J
\mathscr{E}_\la(\np u_n^j(0),\np N_n^j(0))-\mathscr{E}_\la(\pe
u_n(0),\pe N_n(0))=0,}
we have \EQ{\label{energy sum} \sum_{j=1}^J \mathscr{E}_\la(\np
u_n^j(0),\np
N_n^j(0))\leq\lim_{n\rightarrow\infty}\mathscr{E}_\la(u_n,N_n)=E_\la^*.
}

If $\mathscr{E}_\la(\np u_n^j(0),\np N_n^j(0))<E_\la^*$ for all $j<\bar J$, 
 then we have $\|(\np u_n^j,\np N_n^j)\|_{S}<\I$ for all $j$, and so by Lemma
\ref{NL profile}, 
\EQ{\limsup_{n\to\I}\|(u_n,N_n)\|_{S}<\I,} 
which contradicts $\lim_{n\rightarrow\infty}\|(u_n,N_n)\|_{S(0,\I)}=\I$.
Thus, we must have one $j<\ba J$ such that 
 \EQ{
 \mathscr{E}_\la(\np u_n^j(0),\np N_n^j(0))=E_\la^*. }
Without losing generality, we may assume $j=1$. 
Comparing this with \eqref{energy sum}, we have \EQ{
(u_n(0),N_n(0))=U(-t_n)(\lp f^1,\lp g^1)+(u_n^{>1}(0),N_n^{>1}(0)) } and
\EQ{ \|(u_n^{>1}(0),N_n^{>1}(0))\|_{H^1\times L^2}\lesssim
\mathscr{E}_\la(u_n^{>1}(0),N_n^{>1}(0))\rightarrow 0. }

If $t_n \to -\I$, then we have 
\EQ{ \|U(t-t_n)(\lp f^1,\lp g^1)\|_{Z(0,\I)}
 \to 0,} 
and hence
\EQ{ 
 \pt\|U(t)(u_n(0),N_n(0))\|_{Z(0,\I)}
 \pr\lec \|U(t-t_n)(\lp f^1,\lp g^1)\|_{Z(0,\I)}+\|(u_n^{>1}(0),N_n^{>1}(0))\|_{H^1\times L^2}\to 0.}
By Lemma \ref{B-Q-T}, 
\EQ{ 
 \|U(t)B((u_n(0),N_n(0)))\|_{S(0,\I)}+\|NL(U(t)(u_n(0),N_n(0)))\|_{S(0,\I)}\to 0.} 
Then using Lemma \ref{Stability} (with $\vec u_a:=U(t)(u_n(0),N_n(0))$ and $(u_n(0),N_n(0))$ as the initial data), 
we obtain 
\EQ{ \lim_{n\rightarrow\I} \|(u_n,N_n)\|_{S(0,\I)}  < \I. }
which contradicts $\|(u_n,N_n)\|_{S(0,\I)} \to \I$.

If $t_n\rightarrow+\infty$, the argument is similar and we obtain a contradiction by using $\|(u_n,N_n)\|_{S(-\I,0)} \to \I$.

So, the only case left is $t_n \to 0$. In this case, 
\EQ{
 \|(u_n(0),N_n(0))-(\lp f^1,\lp g^1)\|_{H^1\times L^2} \to 0.} 
Let $(u,N)$  be the global solution with initial data
$(u(0),N(0))=(\lp f^1,\lp g^1)$, then $\mathscr{E}_\la(u,N)\leq
E_\la^*$. By stability, we must have 
\EQ{
  \|(u,N)\|_{S(-\I,0)}=\|(u,N)\|_{S(0,\I)}  = \I,} 
since otherwise $(u_n,N_n)$ should be bounded either in $S(-\I,0)$ or in $S(0,\I)$. 
By the definition of $E_\la^*$, $\mathscr{E}_\la(u,N)\geq E_\la^*$
and hence $\mathscr{E}_\la(u,N)= E_\la^*$.

 Since $(u,N)$ is locally in $S$, for any $t_n\in\R$, we have
\EQ{
 \|(u,N)\|_{S(-\I,t_n)}=\I=\|(u,N)\|_{S(t_n,\I)}.
} 
Applying the above argument to the sequence $(u_n(t),N_n(t)):=(u(t+t_n),N(t+t_n))$, we see that $(u(t+t_n),N(t+t_n))$ is precompact in $H^1\times L^2$. Thus we obtain the desired result. 
\end{proof}

\section{Rigidity Theorem}

The main purpose of this section is to disprove the existence of
critical element that was constructed in the previous section under the assumption $E_\la^*<J(Q)$. The
main tool is the spatial localization of the virial identity. We
 prove

\begin{thm}[Rigidity Theorem]
Let $(u,N)$ be a global solution to \eqref{eq:Zak2} satisfying
$K(u)\geq 0$, and $E_Z(u,N)+\lambda^2 M(u)< J_\lambda (Q_\lambda)$
for some $\lambda>0$. Moreover, assume $\{(u,N)(t):t\in \R\}$ is
precompact in $H^1\times L^2$. Then $u=N\equiv 0$.
\end{thm}

\begin{proof}
By contradiction, we assume $(u,N)\ne (0,0)$. Then by the
compactness we may assume further $u\ne 0$, since otherwise $N$ would be a free wave and dispersive. We divide the proof into
the following three steps:

{\bf Step 1:} Energy trapping.

We claim that 
\EQ{
  c:=\inf_{t\in \R}K(u)>0.}
If not, then there exists $\{t_n\}$ with $t_n\to t_*\in [-\I,\I]$,
and $K(u(t_n))\to 0$. By the precompactness of $\{u(t):t\in \R\}$, we
get that up to a sequence $(u(t_n),N(t_n))$ converges to some $(f,g)$ in $H^1\times L^2$. Then we have $K(f)=0$, $J_\la(f)\leq
E_Z(f,g)+\la^2 M(f)=E_Z(u,N)+\la^2 M(u)<J_\la(Q_\la)$. By the
variational characterization of $Q_\lambda$, we get $f\equiv 0$
which contradicts to the $M(f)=M(u)\ne 0$.

{\bf Step 2:} Uniform small tails.

Let $\nu=\Re N-|u|^2=n-|u|^2$. We claim that for any $\e>0$, there
exists $R>0$ such that at any $t\in \R$, we have
\begin{align*}
\int_{|x|\geq R}\big(|\nabla
u|^2+|u|^2+|u|^4+|u|^6+|\nu|^2+|D^{-1}\nabla
\nu|^2+\frac{|D^{-1}\nu|^2}{|x|^2}\big)dx<\e.
\end{align*}
Indeed, since $\{(u,N)(t):t\in \R\}$ is precompact in $H^1\times
L^2$, by Sobolev embedding and the $L^p$-boundedness of
$D^{-1}\nabla$, we get that $\{u(t)\}$ is precompact in
$L^2,L^4,L^6$, $\{D^{-1}N(t)\}$ is precompact in $\dot{H}^1$, and
$\{D^{-1}\nabla N(t),\nu(t),D^{-1}\nabla \nu\}$ is precompact in
$L^2$. Then the claim follows immediately.

{\bf Step 3:} Contradiction to the local virial estimates.

We recall the local virial estimates obtained in Section
\ref{sec:growup}. For any $R>0$ \EQ{
 V_R(t):=\LR{\UJ u|(AX+XA)u}.}
where $X=X^*$ be the operator of smooth trancation to $|x|<R$ by
multiplication with $\psi_R(x)$. From the proof in Section
\ref{sec:growup} and Corollary \ref{cor:cor} we have
\begin{align}\label{eq:VR bound}
|V_R(t)| \lec R[\|u\|_2\|\na u\|_2+\|N\|_2^2]\lec R.
\end{align}
On the other hand, from Step 2, Step 1, and Lemma \ref{lem:var},
we get
\begin{align*}
V_R'(t)=&\LR{\nu|X\nu}/2 + \LR{\na\y|X\na\y}/2 +
\LR{\y_r|r\psi_R'\y_r}-\frac14\LR{|\y|^2|A_{1}\De\psi_R}\\
&-2\LR{\nu||u|^2}+O(R^{-1}\|\nu\|_2\|u\|_2^{3/2}\|\na
u\|_2^{1/2})\\
&+\LR{\na u|4X\na u}+4\LR{u_r|\psi_R'ru_r}\\
&-3\|u\|_4^4 +O(R^{-2}(\|u\|_2^2+\|u\|_2^3\|\na u\|_2)) \quad \mbox{(obtained in Section \ref{sec:growup})}\\
=&4K(u)+\|\nu\|_2^2-2\LR{\nu||u|^2}+o(1),\ R\to \infty \quad
\mbox{(by Step 2)}\\
\geq&(1-\frac{2}{\sqrt{6}})K(u)+o(1)\geq c/2, \ R\gg 1. \quad
\mbox{(by Step 1 and Lemma \ref{lem:var})}
\end{align*}
Thus we get
\[V_R(t)\geq V_R(0)+c t/2,\]
which contradicts \eqref{eq:VR bound} for sufficiently large $t$.
\end{proof}

\appendix

\section{Construction of wave operators}
Here we briefly skecth a proof for the existence of the wave operators, or the solvability of the final state problem. 
For the construction of a nonlinear profile in the radial setting,  we need only to consider a sequence of solutions in the form
\EQ{
\vec u_n=U(t)\vec f + \int_{-t_n}^t U(t-s)(nu, \alpha D|u|^2)ds,}
with $t_n\to\pm\I$, which is normally transformed into such a form as
\EQ{\label{u_n-wo}
 \vec u_n=&U(t)\vec f-U(t+t_n)B(U(-t_n)\vec f)+B(\vec u_n)
 +Q_{-t_n}(\vec u_n)+T_{-t_n}(\vec u_n),}
where $Q_{-t_n}$ and $T_{-t_n}$ denote respectively $Q$ and $T$ with the Duhamel integration $\int_{-t_n}^t$, and arbitrarily fixed $\be$, say $\be=10$. 
Below we consider only the case $t_n\to-\I$, since the other case is similar. 
The following is the precise statement that we need for the nonlinear profile in this case. 
\begin{lem}\label{Exist}
Let $\vec f\in H^1\times L^2$, $\R\ni t_n\to-\I$, and let $\{\vec u_n\}$ be the sequence of solutions to the Zakharov system with the Cauchy data $\vec u_n(-t_n)=U(-t_n)\vec f$. Then there exist $T\in\R$ and a unique $\vec u\in S(T,\I)$ satisfying 
\EQ{ \label{wave op}
 \pt \vec u=U(t)\vec f+B(\vec u)+Q_{\I}(\vec u)+T_{\I}(\vec u),
 \pq \lim_{n\to\I}\|\vec u_n-\vec u\|_{S(T,\I)}=0,}
as well as the Zakharov system on $(T,\I)$. Moreover, if $\{\vec u_n\}$ is bounded in $L^\I(\R;H^1\times L^2)$, then $\vec u$ is global and the above convergence holds for any $T\in\R$. 
\end{lem}
\begin{proof}
First, we can solve \eqref{wave op} on $(T,\I)$ for $T\gg 1$, by the iteration argument similar to \cite{GN} in the space 
\EQ{
 X:=\{\vec u\in C([T,\I);H^1\times L^2)\ |\ \|\vec u\|_{S(T,\I)} \lec 1, \|\vec u\|_{Z(T,\I)} \leq \eta\},}
with $\y:=2\|U(t)\vec f\|_{Z(T,\I)}\ll 1$, using the estimates similar to \eqref{NL est} as well as 
\EQ{
\|U(t)\vec f\|_{Z(T,\I)}\lesssim \|U(t)\vec f\|_{L^\I_{t>T}(L^6_x \times \dot{H}^{-1}_6)}\to 0 \text{ as } T\to\I.}
Similar estimates imply that $\vec u_n$ are scattering as $t\to\I$ for large $n$. Also similarly to \eqref{NL est}, we have for some $\te>0$
\EQ{
 \|U(t+t_n)B(U(-t_n)\vec f)\|_{S} \pt\lec \|B(U(-t_n)\vec f)\|_{H^1\times L^2} 
 \pr\lec \|U(-t_n)\vec f\|_{H^1\times L^2}^{1-\te}\|U(-t_n)\vec f\|_{L^6\times \dot H^{-1}_6}^\te \to 0.}
Then by applying \eqref{NL est} to the difference equation, we obtain the convergence $\vec u_n\to \vec u$ in $S(T,\I)$. Since $\vec u_n$ solves the Zakharov system, so does the limit $\vec u$. If the former is uniformly bounded in $H^1\times L^2$, so is the latter, and the convergence is also extended to arbitrary $(T,\I)$ by the local wellposedness. 
\end{proof}

\subsection*{Acknowledgment}
Z. Guo is supported in part by NNSF of China (No. 11001003) and RFDP
of China (No. 20110001120111). 
S. Wang is supported by China Scholarship Council.


\begin{thebibliography}{10}
\bibitem{BG} H.~Bahouri and P.~G\'erard, {\it High frequency approximation of solutions to critical nonlinear wave equations.} Amer. J. Math. {\bf 121} (1999), no.~1, 131--175. 

\bibitem{BeHe} I.~Bejenaru and S.~Herr, {\it Convolutions of singular measures and applications to the Zakharov system.} J. Funct. Anal. {\bf 261} (2011), no.~2, 478--506.

\bibitem{BHHT} I.~Bejenaru, S.~Herr, J.~Holmer, and D.~Tataru, {\it On the 2D
Zakharov system with $L^2$ Schr\"odinger data.} Nonlinearity {\bf 22} 
(2009), no.~5, 1063--1089.

\bibitem{BoCo} J.~Bourgain and J.~Colliander, {\it On wellposedness of the Zakharov system.} Internat. Math. Res. Notices {\bf 1996}, 
no.~11, 515--546.

\bibitem{GV} J.~Ginibre and G.~Velo, {\it Scattering theory for the Zakharov system.} Hokkaido Math. J. {\bf 35} (2006), no.~4, 865--892.

\bibitem{GTV} J.~Ginibre, Y.~Tsutsumi, and G.~Velo, {\it On the Cauchy problem for the Zakharov system.} J. Funct. Anal. {\bf 151} (1997), no.~2, 384--436.


\bibitem{GN} Z.~Guo and K.~Nakanishi, {\it Small energy scattering for the Zakharov system with radial symmetry.} arXiv:1203.3959v1 [math.AP]. 


\bibitem{HR} J.~Holmer and S.~Roudenko, {\it A sharp condition for scattering of the radial 3D Cubic nonlinear Schr\"{o}dinger equation.} Commun. Math. Phys. 
{\bf 282} (2008), 435--467.

\bibitem{DHR} T.~Duyckaerts, J.~Holmer and S.~Roudenko, {\it Scattering for the non-radial 3d cubic nonlinear Schr\"odinger equation.} Math. Res. Lett. {\bf 15} (2008), 1233--1250.

\bibitem{KM} C.~E.~Kenig and F.~Merle, {\it Global well-posedness, scattering and blow-up for the energy-critical, focusing, non-linear Schr\"odinger equation in the radial case.} Invent. Math. {\bf 166} (2006), no.~3, 645--675. 

\bibitem{KPV} C.~Kenig, G.~Ponce and L.~Vega, {\it On the Zakharov and Zakharov-Schulman systems.} J. Funct. Anal. {\bf 127} (1995), no.~1, 204--234.

\bibitem{Kishi} N.~Kishimoto, {\it Local well-posedness for the Zakharov system on multidimensional torus.} preprint (2011). 

\bibitem{MN} N.~Masmoudi and K.~Nakanishi, {\it Energy convergence for singular limits of Zakharov type systems.} Invent. Math. {\bf 172} (2008), no.~3, 535--583.

\bibitem{Merle} F.~Merle, {\it Blow-up results of virial type for Zakharov equations.} Comm. Math. Phys. {\bf 175} (1996), 433--455.

\bibitem{OgT} T.~Ogawa and Y.~Tsutsumi, 
{\it Blow-up of $H^1$ solution for the nonlinear Schr\"odinger equation. J. Differential Equations.} {\bf 92} (1991), no.~2, 317--330. 

\bibitem{OT} T.~Ozawa and Y.~Tsutsumi, {\it The nonlinear Schr\"odinger limit and the initial layer of the Zakharov equations.} Differ. Integral Equ. {\bf 5} (1992) no.~4, 721--745.

\bibitem{OT2} T.~Ozawa and Y.~Tsutsumi, {\it Global existence and asymptotic behavior of solutions for the Zakharov equations in three-dimensions space.} Adv. Math. Sci. Appl. {\bf 3} (Special Issue) (1993/94) 301--334.

\bibitem{SW} S.~Schochet and M.~Weinstein, {\it The nonlinear Schr\"odinger limit of the Zakharov equations governing Langmuir turbulence.} Commun. Math. Phys. {\bf 106} (1986), no.~4, 569--580.

\bibitem{Shimo} A.~Shimomura, {\it Scattering theory for Zakharov equations in three-dimensional space with large data.} Commun. Contemp. Math. {\bf 6} (2004), no.~6, 881--899.

\bibitem{Takaoka} H.~Takaoka, {\it Well-posedness for the Zakharov system with the periodic boundary condition.} Differential Integral Equations {\bf 12} (1999), no.~6, 789--810.

\bibitem{Zak} V.~E.~Zakharov, {\it Collapse of Langmuir waves.} Sov. Phys. JETP  {\bf 35} (1972), 908--914.

\end{thebibliography}
\end{document}